\documentclass[10pt,reqno]{amsart}
\usepackage{amssymb,mathrsfs,color}

\usepackage{wrapfig}

\usepackage{tikz}
\usepgflibrary{fpu}
\usetikzlibrary{arrows,calc,decorations.pathreplacing}
\definecolor{light-gray1}{gray}{0.90}
\definecolor{light-gray2}{gray}{0.80}

\definecolor{deepgreen}{cmyk}{1,0,1,0.5}

\newcommand{\E}{\mathcal{E}}

\newcommand{\J}{\mathcal{J}}

\newcommand{\HH}{\mathcal{H}}

\newcommand{\Sp}{\mathbb{S}}
\newcommand{\EE}{\mathscr{E}}
\newcommand{\FF}{\mathscr{F}}
\newcommand{\CC}{\mathscr{C}}

\newcommand{\N}{\mathbb{N}}
\newcommand{\R}{\mathbb{R}}
\newcommand{\Z}{\mathbb{Z}}

\newcommand{\al}{\alpha}
\newcommand{\be}{\beta}
\newcommand{\ga}{\gamma}
\newcommand{\de}{\delta}
\newcommand{\e}{\varepsilon}
\newcommand{\fy}{\varphi}
\newcommand{\om}{\omega}
\newcommand{\la}{\lambda}

\newcommand{\s}{\sigma}

\newcommand{\Om}{\Omega}

\newcommand{\p}{\partial}

\newcommand{\supp}{\operatorname{supp}}

\newcommand{\loc}{\operatorname{loc}}

\newcommand{\I}{\infty}

\newcommand{\ti}{\widetilde}

\newcommand{\ds}{\displaystyle}

\newcommand{\abs}[1]{\left\lvert{#1}\right\rvert}

\newcommand{\Del}[1]{}

\def\ti{\tilde}

\numberwithin{equation}{section}

\newtheorem{thm}{Theorem}[section]
\newtheorem{cor}[thm]{Corollary}
\newtheorem{lem}[thm]{Lemma}
\newtheorem{prop}[thm]{Proposition}

\theoremstyle{remark}
\newtheorem{rem}{Remark}

\begin{document}

\author{R.~C\^{o}te}
\author{C.~E.~Kenig}
\author{A.~Lawrie}
\author{W.~Schlag}

\title[Large energy solutions of the equivariant wave map problem: II]{Characterization of large energy solutions of the equivariant wave map problem: II}
\begin{abstract} We consider $1$-equivariant wave maps from $\R^{1+2} \to \Sp^2$ of finite energy. We establish a classification of all degree one global solutions  whose energies 
are  less than three times the energy of the harmonic map~$Q$. In particular, for each global energy solution of topological degree one, we show that  the solution asymptotically decouples  into a rescaled harmonic map plus a radiation term. 
Together with the companion article, \cite{CKLS1}, where we consider the case of finite-time blow up, this gives a  characterization of all $1$-equivariant, degree~$1$ wave maps in the energy regime~$[E(Q), 3E(Q))$.  
\end{abstract}

\thanks{Support of the National Science Foundation DMS-0968472 for the second author, and  DMS-0617854, DMS-1160817 for the fourth author is gratefully acknowledged. 
This first author wishes to thank the University of Chicago for its hospitality during the academic year 2011-12, and acknowledges support from the European Research Council through the project BLOWDISOL}



\maketitle

\section{Introduction}
This paper is the companion article to \cite{CKLS1}. Here we continue our study of the equivariant wave maps problem from $1+2$ dimensional Minkowski space to $2$--dimensional surfaces of revolution.
%in particular the  $2$--sphere, $\Sp^2$.
 In local coordinates on the target manifold, $(M, g)$,  the Cauchy problem for wave maps is given by
\begin{align} \label{cp i}
  &\Box U^k =  - \eta^{\al \be} \Gamma^k_{i j}(U) \p_{\al} U^i \p_{\be} U^j\\
  &(U, \p_t U)\vert_{t=0}=(U_0, U_1), \notag
  \end{align}
 where $\Gamma_{i j}^k$ are the Christoffel symbols on $TM$. As in \cite{CKLS1} we will, for simplicity, restrict our attention to the case when the target $(M,g)= (\Sp^2, g)$ with $g$ the round metric on the $2$--sphere, $\Sp^2$. Our results here apply to more general compact surfaces of revolution as well, and we refer the reader to \cite[Appendix $A$]{CKLS1} for more details. 

In  spherical coordinates, $$(\psi, \om) \mapsto (\sin \psi \cos \om , \sin \psi \sin \om, \cos \psi),$$ on $\Sp^2$, the metric, $g$,  is given by the matrix $g = \textrm{diag}( 1, \sin^2(\psi))$. In the case of  $1$-equivariant wave maps, we require our wave map, $U$, to have the form 
\begin{align*}
U(t, r, \om) = (\psi(t, r), \om) \mapsto (\sin \psi(t, r) \cos \om ,\, \sin \psi(t, r) \sin \om, \,\cos \psi(t, r)),
 \end{align*}
 where $(r, \om)$ are polar coordinates on $\R^2$. In this case, the Cauchy problem \eqref{cp i} reduces to
\begin{align} \label{cp} 
&\psi_{tt} - \psi_{rr} - \frac{1}{r} \psi_r + \frac{\sin(2\psi)}{2r^2} = 0\\
&(\psi, \psi_t)\vert_{t=0} = (\psi_0, \psi_1). \notag
\end{align} 
Wave maps exhibit a conserved energy, which in this equivariant setting is given by 
\begin{align*}
\E(U, \p_tU)(t)=\E(\psi, \psi_t)(t) = \int_0^{\infty} \left(\psi_t^2 + \psi_r^2 + \frac{\sin^2(\psi)}{r^2}\right) \, r\, dr = \textrm{const.,} 
\end{align*}
and they are invariant under the scaling $$\vec \psi(t, r):= (\psi(t, r), \psi_t(t, r)) \mapsto (\psi(\la t, \la r), \la \psi_t(\la t \la r)).$$ 
The conserved energy is also invariant under this scaling which means that the Cauchy problem under consideration is energy critical. 

We refer the reader to \cite{CKLS1} for a more detailed introduction and history of the equivariant wave maps problem.

As in \cite{CKLS1}, we note that any wave map $\vec \psi(t, r)$ with finite energy and continuous dependence on $t \in I$ satisfies $\psi(t, 0)=m \pi$ and $\psi(t, \infty)= n \pi$ for all $t \in I$ for fixed integers $m, n$. This determines a disjoint set of energy classes
\begin{align}\label{Hnm}
 \HH_{m, n} := \{ (\psi_0, \psi_1) \,\vert\,  \E(\psi_0, \psi_1) < \infty \quad \textrm{and} \quad \psi_0(0) =m\pi, \, \psi_0(\infty) = n\pi\}. 
 \end{align}
We will mainly consider the spaces $\HH_{0, n}$ and we denote these by $\HH_n := \HH_{0, n}$. In this case we refer to $n$ as the degree of the map. We also define $\HH= \bigcup_{n\in \Z} \HH_{n}$ to be the full energy space.

In our analysis, an important role is played by the unique (up to scaling) non-trivial harmonic map, $Q(r)= 2 \arctan(r)$, given by stereographic projection.  We note that $Q$ solves 
\begin{align}\label{Q}
Q_{rr} + \frac{1}{r}Q_r = \frac{\sin(2Q)}{2r^2}.
\end{align} 
Observe in addition that  $(Q, 0) \in \HH_1$ and in fact $(Q, 0)$ has minimal energy in $\HH_1$ with $\E(Q):= \E(Q,0)=4$. Note the slight abuse of notation above in that we will denote the energy of the element $(Q, 0) \in \HH_1$ by $\E(Q)$ rather than $\E(Q, 0)$.

Recall that in \cite{CKLS1} we showed that for any data $\vec \psi(0)$ in the zero topological class, $\HH_0$, with energy $\E( \vec \psi) <2\E(Q)$ there is a corresponding unique global wave map evolution $\vec\psi(t, r)$ that scatters to zero in the sense that the energy of $\vec \psi(t)$ on any arbitrary, but fixed compact region vanishes as $t \to  \infty$, see \cite[Theorem $1.1$]{CKLS1}. An equivalent way to view this scattering property is that there exists a decomposition 
\begin{align}\label{deg 0 scat}
\vec \psi(t)= \vec \fy_L(t) + o_{\HH}(1) \quad \textrm{as} \quad t \to \infty
\end{align}
where $\vec \fy_L(t) \in \HH_0$ solves the linearized version of \eqref{cp}:
\begin{align} \label{2d lin}
\fy_{tt} -\fy_{rr} -\frac{1}{r} \fy_r + \frac{1}{r^2} \fy = 0
\end{align} 
This result was proved via the concentration-compactness/rigidity method which was developed by the second author and Merle in \cite{KM06} and  \cite{KM08}, and it provides a complete classification of all solutions in $\HH_0$ with energy below $2\E(Q)$, namely, they all exist globally and scatter to zero. We note that this theorem is also a consequence of the work by Sterbenz and Tataru in \cite{ST2} if one considers their results in the equivariant setting. 

In \cite{CKLS1} we also study {\em degree one} wave maps, $\vec \psi(t) \in \HH_1$, with energy $\E(\vec \psi) = \E(Q) + \eta <3 \E(Q)$  that blow up in finite time. Because we are working in the equivariant, energy critical setting, blow-up can only occur  at the origin in $\R^2$ and in an energy concentration scenario. We show that if blow-up does occur, say at $t=1$, then there exists a scaling parameter $\la(t)=o(1-t)$, a degree zero map $\vec \fy \in \HH_0$ and a decomposition 
\begin{align}\label{bu decomp}
\vec \psi(t, r) =   \vec \fy(r) + \left(Q\left(r/\la(t)\right), 0\right) + o_{\HH}(1) \quad \textrm{as} \quad t \to 1.
\end{align}

Here we complete our study of {\em degree one} solutions to \eqref{cp}, i.e., solutions that lie in $\HH_1$, with energy below $3\E(Q)$, by providing a classification of such solutions with this energy constraint. Since the degree of the map is preserved for all time, scattering to zero is not possible for a degree one solution. However, we show that a decomposition of the form \eqref{bu decomp} holds in the global case. In particular we establish the following  theorem:

%-----------------------------------------classification theorem-------------------------------------------------------%

\begin{thm}[Classification of  solutions in $\HH_1$ with energies below $3\E(Q)$]\label{glob sol} Let $\vec \psi(0)\in \HH_1$ and denote by $\vec \psi(t) \in \HH_1$ the corresponding wave map evolution. Suppose that $\vec \psi$ satisfies 
\begin{align*}
\E(\vec \psi) = \E(Q) + \eta < 3\E(Q).
\end{align*}
Then, one of the following two scenarios occurs:
\begin{itemize}
\item[$(1)$] \textbf{Finite time blow-up:} The solution $\vec \psi(t)$ blows up in finite time, say at $t=1$, and there exists a continuous function, $\la:[0,1) \to (0, \infty)$ with $\la(t) = o(1-t)$, a map $\vec \fy=(\fy_0, \fy_1) \in \HH_0$ with $\E(\vec \fy)= \eta$, and a decomposition  
\begin{align}\label{dec1}
\vec \psi(t) =   \vec \fy + \left(Q\left(\cdot/\la(t)\right), 0\right) + \vec\epsilon(t)
\end{align}
such that $\vec \epsilon(t) \in \HH_0$ and $\vec \epsilon(t)\to 0$ in $\HH_0$ as $t \to 1$. 
\item[$(2)$] \textbf{Global Solution:} The solution $\vec \psi(t) \in \HH_1$ exists globally in time and there exists a continuous function, $\la:[0,\infty) \to (0, \infty)$ with $\la(t) = o(t)$ as $t \to \infty$, a solution $\vec \fy_L(t) \in \HH_0$ to the linear wave equation \eqref{2d lin}, %with $\E(\vec \fy_L)= \eta$, 
and a decomposition  
\begin{align}\label{dec}
\vec \psi(t) =   \vec \fy_L(t) + \left(Q\left(\cdot/\la(t)\right), 0\right) + \vec\epsilon(t)
\end{align}
such that $\vec \epsilon(t) \in \HH_0$ and $\vec \epsilon(t)\to 0$ in $\HH_0$ as $t \to \infty$. 
\end{itemize}
\end{thm}

%-----------------------------------------classification theorem-------------------------------------------------------%

\begin{rem} One should note that the requirement $\la(t)=o(t)$ as $ t \to \infty$ in part $(2)$ above leaves open many possibilities for the asymptotic behavior of global degree one solutions to \eqref{cp} with energy below $3\E(Q)$. %In particular there are three distinct possibilities which are determined by the behavior of $\la(t)$. 
If $\la(t) \to \la_0 \in (0, \infty)$  then our theorem says that the solution $\psi(t)$ asymptotically decouples into a soliton, $Q_{\la_0}$, plus a purely dispersive term, and one can call this {\em scattering to $Q_{\la_0}$}. If $\la(t) \to 0$ as $t \to \infty$ then this means that the solution is concentrating $\E(Q)$ worth of energy at the origin  as $t \to\infty$ and we refer to this phenomenon as {\em infinite time blow-up}. Finally, if  $\la(t) \to \infty$ as $t \to \infty$ then the solution can be thought of as concentrating $\E(Q)$ worth of energy  at spacial infinity as $t \to \infty$ and we call this {\em infinite time flattening}.  %Here the distinction between the words  {\em blow-up} and  {\em flattening} is slightly misleading, however, as the only real difference is that in the former case the wave map concentrates $\E(Q)$ worth of energy at the north pole of the sphere, and in the latter case the wave map concentrates $\E(Q)$ worth of energy at the south pole in infinite time. 

We also would like to highlight the fact that global solutions of the type mentioned above, i.e., infinite time blow-up and flattening,  have been constructed in the case of the $3d$ semi-linear focusing energy critical wave equation by Donninger and Krieger in \cite{DK}. No constructions of this type are known at this point for the energy critical wave maps studied here.  In addition, a classification of all the possible dynamics for maps in $\HH_1$ at energy levels $\ge 3\E(Q)$ remains open. 
\end{rem} 

\begin{rem}We emphasize that \cite{CKLS1} goes hand-in-hand with this article and the two papers are intended to be read together. In fact,  part $(1)$ of Theorem~\ref{glob sol} was established in \cite[Theorem~$1.3$]{CKLS1}.  Therefore, in order to complete the proof of Theorem~\ref{glob sol} we need to prove only part $(2)$ and the rest of this paper will be devoted to that goal. The broad outline of the proof of Theorem~\ref{glob sol} $(2)$ is similar in nature to the proof of part $(1)$. With this is mind we will often refer the reader to \cite{CKLS1} where the details are nearly identical instead of repeating the same arguments here. 
\end{rem}

\begin{rem} We remark that Theorem~\ref{glob sol} is reminiscent of the recent works of Duyckaerts, the second author, and Merle in \cite{DKM1, DKM2, DKM3, DKM4} for the energy critical semi-linear focusing wave equation in $3$ spacial dimensions and again we refer the reader to \cite{CKLS1} for a more detailed description of the similarities and differences between these papers.  
\end{rem}
 
 \begin{rem} Finally, we would like to note that the same observations in \cite[Appendix~A]{CKLS1} regarding $1$-equivariant wave maps to more general targets, higher equivariance classes and the $4d$ equivariant Yang-Mills system hold in the context of the global statement in Theorem~\ref{glob sol}. 
 \end{rem} 
 
%\begin{thm}[Classification of global solutions in $\HH_1$ with energies below $3\E(Q)$]\label{glob sol} 

%Let $\vec \psi(t) \in \HH_1$  be a smooth, global solution to  \eqref{cp}  with  $$\E(\vec \psi) = \E(Q) + \eta < 3\E(Q). $$ Then, there exists a continuous function, $\la:[0,\infty) \to (0, \infty)$ with $\la(t) = o(t)$ as $t \to \infty$, a solution $\vec \fy_L(t) \in \HH_0$ to the linear wave equation \eqref{2d lin}, %with $\E(\vec \fy_L)= \eta$, 
%and a decomposition  
%\begin{align}\label{dec}
%\vec \psi(t) =   \vec \fy_L(t) + \left(Q\left(\cdot/\la(t)\right), 0\right) + \vec\epsilon(t)
%\end{align}
%such that $\vec \epsilon(t) \in \HH_0$ and $\vec \epsilon(t)\to 0$ in $\HH_0$ as $t \to \infty$. In fact, we show that $\vec \epsilon(t)\to 0$ in $H \times L^2$ as $t \to \infty$.
%\end{thm}

\section{Preliminaries}
For the reader's convenience, we recall a few facts and notations from \cite{CKLS1} that are used frequently in what follows. 
We define the $1$-equivariant energy space  to be $$\HH= \{\vec U \in \dot{H}^1 \times L^2(\R^2; \Sp^2) \, \vert\, U \circ \rho = \rho \circ U, \,\,\, \forall \rho \in SO(2)\}.$$ $\HH$ is endowed with the norm 
\begin{align}\label{en norm}
\E(\vec U(t)) = \|\vec U(t) \|_{\dot{H}^1 \times L^2(\R^2; \Sp^2)}^2 = \int_{\R^2}  ( \abs{\p_t U}^2_g +  \abs{\nabla U}_g^2 ) \, dx.
\end{align}
As noted in the introduction, by our equivariance condition we can write $U(t, r, \om) = (\psi(t,r), \om)$ and the energy of a wave map becomes 
\begin{align}\label{en psi}
\E(U, \p_tU)(t)=\E(\psi, \psi_t)(t) = \int_0^{\infty} \left(\psi_t^2 + \psi_r^2 + \frac{\sin^2(\psi)}{r^2}\right) \, r\, dr = \textrm{const.} 
\end{align}
We also define the localized energy as follows: Let $r_1, r_2 \in [0, \infty)$. Then   
\begin{align*}
\E_{r_1}^{r_2}(\vec \psi(t)):= \int_{r_1}^{r_2}  \left(\psi_t^2 + \psi_r^2 + \frac{\sin^2(\psi)}{r^2}\right) \, r\, dr.
\end{align*}
Following Shatah and Struwe, \cite{SS}, we set
\begin{align}\label{G def}
G(\psi):= \int_0^{\psi} \abs{\sin \rho} \, d\rho.
\end{align}
Observe that for any $(\psi, 0) \in \HH$ and for any $r_1, r_2 \in[0, \infty)$ we have 
\begin{align}\label{b R}
\abs{G(\psi(r_2))-G(\psi(r_1))} &=\abs{\int_{\psi(r_1)}^{\psi(r_2)} \abs{\sin \rho} \, d\rho} 
\\
&= \abs{\int_{r_1}^{r_2} \abs{\sin(\psi(r))} \psi_r(r) \, dr} \notag
 \le\frac{1}{2} \E_{r_1}^{r_2}(\psi, 0).
\end{align}
We also recall from \cite{CKLS1} the definition of the space $H \times L^2$. 
\begin{align}\label{HxL^2}  
\|(\psi_0, \psi_1)\|_{H \times L^2}^2 := \int_0^{\infty} \left(\psi_1^2 + (\psi_0)_r^2 + \frac{ \psi_0^2}{r^2} \right) \, r\, dr.
\end{align} 
We note that for degree zero maps $(\psi_0, \psi_1) \in \HH_0$ the energy is comparable to the $H \times L^2$ norm provided the $L^{\infty}$ norm of $\psi_0$ is uniformly bounded below $\pi$.  This equivalence of norms is detailed in \cite[Lemma $2.1$]{CKLS1}, see also \cite[Lemma~$2$]{CKM}. The space $H \times L^2$ is not defined for maps  $(\psi_0, \psi_1) \in \HH_1$, but one can instead consider the $H \times L^2$ norm of $(\psi_0-Q_{\la}, 0)$ for $\la \in (0, \infty)$, and $Q_{\la}(r)= Q(r/ \la)$. In fact, for maps $\vec \psi \in \HH_1$ such that $\E(\vec \psi)- \E(Q)$ is small, one can choose $\la>0$ so that  
\begin{align*}
\|(\psi_0- Q_{\la}, \psi_1)\|_{H \times L^2}^2 \simeq \E(\vec \psi)- \E(Q).
\end{align*}
This amounts to the coercivity of the energy near~$Q$ up to the scaling symmetry. 
For more details we refer the reader to \cite[Proposition~$4.3$]{Co}, \cite[Lemma~$2.5$]{CKLS1}, and~\cite{BKT}.

\subsection{Properties of global wave maps}

We will need a few facts about global solutions to  \eqref{cp}. The results in this  section constitute slight refinements and a few consequences of the work of Shatah and Tahvildar-Zadeh in \cite[Section~$3.1$]{STZ} on global equivariant wave maps and originate in the work of Christodoulou and Tahvildar-Zadeh on spherically symmetric wave maps, see~\cite{CTZ1}. 

\begin{prop}\label{ctzcor1}
Let $\vec \psi(t) \in \HH$ be a global wave map. Let $0 < \la<1$. Then we have 
\begin{align}\label{limsup la}
\limsup_{t \to \infty}\E_{\la t}^{t-A}(\vec \psi(t)) \to 0 \quad \textrm{as} \quad A \to \infty.
\end{align}
In fact,  we have 
\begin{align}\label{limlim}
\E_{\la t}^{t-A}(\vec \psi(t)) \to 0 \quad \textrm{as}\quad  t, \, A \to \infty \quad \textrm{for} \quad A \le (1- \la)t.
\end{align}
\end{prop}
We note that Proposition~\ref{ctzcor1} is a refinement of  \cite[$(3.4)$]{STZ}, see also \cite[Corollary $1$]{CTZ1} where the case of spherically symmetric wave maps is considered. To prove this result, we follow \cite{CTZ1}, \cite{STZ}, and \cite{SS} and introduce the following quantities: 
\begin{align*}
&e(t,r):= \psi_t^2(t, r) + \psi_r^2(t, r) + \frac{\sin^2(\psi(t,r))}{r^2}\\
&m(t,r):= 2 \psi_t(t, r) \psi_r(t, r).
\end{align*}
Observe that with this notation the energy identity becomes: 
\begin{align}
\p_t e(t, r)= \frac{1}{r} \p_r\left( r\, m(t,r)\right),
\end{align}
which we can conveniently rewrite as 
\begin{align}\label{en id}
\p_t(r e(t, r))- \p_r(r\,m(t, r))=0.
\end{align}
Using the notation in \cite{CTZ1}, we set 
\begin{align*}
\al^2(t, r):= r\, (e(t,r)+m(t, r))\\
\be^2(t, r):= r(e(t,r)- m(t, r))
\end{align*}
and we define null coordinates 
\begin{align*}
u= t-r, \quad v=t+r.
\end{align*}
Next, for $0 \le \la <1$ set
\begin{align}
&\EE_{\la}(u):= \int_{\frac{1+\la}{1-\la} u}^{\infty} \al^2(u, v) \, dv \label{ee def}\\
&\FF(u_0, u_1):= \lim_{v \to \infty} \int_{u_0}^{u_1} \be^2(u, v) \, du\label{ff def}.
\end{align}
Also, let $\CC_{\rho}^{\pm}$ denote the interior of the forward (resp. backward) light-cone with vertex at $(t, r)=(\rho, 0)$ for $\rho>0$ in $(t, r)$ coordinates.

As in \cite[Section $3.1$]{STZ}, one can show that the integral in \eqref{ee def} and the limit in \eqref{ff def} exist for a wave map of finite energy, see also \cite[Section $2$]{CTZ1} for the details of the argument for the spherically symmetric case. 

By integrating the energy identity \eqref{en id} over the region $(\CC_{u_0}^{+} \backslash \CC_{u_1}^{+}) \cap \CC^{-}_{v_0}$, where $0<u_0<u_1<v_0$, we obtain the identity 
\begin{align*}
\int_{u_0}^{u_1} \be^2(u, v) \, du = \int_{u_0}^{v_0} \al^2(u_0, v) \, dv - \int_{u_1}^{v_0} \al^2(u_1, v) \, dv.
\end{align*}
Letting $v_0 \to \infty$ we see that 
\begin{align}\label{f=e-e}
0 \le \FF(u_0, u_1) = \EE_0(u_0)- \EE_0(u_1),
\end{align}
which shows that $\EE_0$ is decreasing. Next, note that 
$$\FF(u, u_2) = \FF(u, u_1)+ \FF(u_1, u_2) \ge \FF(u, u_1)$$ 
for $u_2>u_1$, and thus $\FF(u, u_1)$ is increasing in $u_1$. $\FF(u, u_1)$ is also bounded above by $\EE(u)$ so 
\begin{align*}
\FF(u):= \lim_{u_1 \to \infty} \FF(u, u_1)
\end{align*}
exists and, as in \cite{STZ}, \cite{CTZ1}, we have 
\begin{align}\label{ff to 0}
\FF(u) \to 0 \quad \textrm{as} \quad u \to \infty.
\end{align}
Finally note that the argument in \cite[Lemma $1$]{CTZ1} shows that for all $0 < \la< 1$ we have  
\begin{align}\label{e la to 0}
\EE_{\la}(u) \to 0 \quad \textrm{as} \quad u \to \infty,
\end{align}
 which is stated in \cite[$(3.3)$]{STZ}. To deduce \eqref{e la to 0}, follow the exact argument in \cite[proof of Lemma~$1$]{CTZ1} using the relevant multiplier  inequalities  for equivariant wave maps established in \cite[proof of Lemma $8.2$]{SS} in place of \cite[equation $(6)$]{CTZ1}. We can now prove Proposition~\ref{ctzcor1}. 

\begin{proof}[Proof of Proposition~\ref{ctzcor1}] Fix $\la \in (0, 1)$ and $\de>0$. Find $A_0$ and $T_0$ large enough so that 
\begin{align*}
0 \le \FF(A) \le \de, \quad 0 \le \EE_{\la}((1-\la) t) \le \de
\end{align*}
for all $A \ge A_0$ and $t \ge T_0$. In $(u, v)$--coordinates consider the points 
\begin{align*}
&X_1= ((1-\la) t, (1+ \la) t), \quad
X_2=(A, 2t-A)\\
&X_3=(A, \bar v), \quad
X_4=((1-\la) t, \bar v)
\end{align*}
where $\bar v $ is very large. Integrating the energy identity \eqref{en id} over the region $\Om$ bounded by the line segments $X_1X_2$, $X_2X_3$, $X_3X_4$, $X_4X_1$ we obtain, 
\begin{align*}
\E_{\la t}^{t-A}( \vec \psi(t)) = - \int_{2t-A}^{\bar v} \al^2(A, v) \, dv + \int_A^{(1-\la)t} \be^2(u, \bar v) \, du + \int_{(1+\la) t}^{\bar v} \al^2((1-\la)t, v) \, dv.
\end{align*}
%--------------------------------------------------------figure------------------------------------------------------------------------------------%
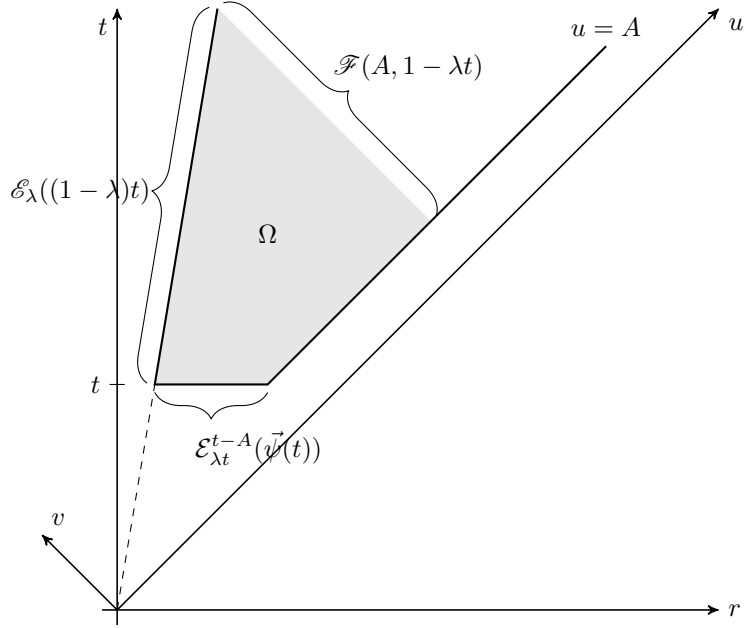
\begin{figure}
\begin{tikzpicture}[
	>=stealth',
	axis/.style={semithick,->},
	coord/.style={dashed, semithick},
	yscale = 1,
	xscale = 1]
	\newcommand{\xmin}{0};
	\newcommand{\xmax}{8};
	\newcommand{\ymin}{0};
	\newcommand{\ymax}{8};
	\newcommand{\ta}{3};
	\newcommand{\lamdba}{6};
	\newcommand{\fsp}{0.2};
	\draw [axis] (\xmin-\fsp,0) -- (\xmax,0) node [right] {$r$};
	\draw [axis] (0,\ymin-\fsp) -- (0,\ymax) node [below left] {$t$};
	\draw [axis] (0,0) -- (\xmax, \xmax) node [below right] {$u$};
	\draw [axis] (0,0) -- (-1,1) node [above right] {$v$};
	\draw [dashed] (0,0) -- (\ta/\lamdba, \ta);
	\draw (-0.1, \ta) node [left] {$t$} -- (0.1,\ta);
	\filldraw[color=light-gray1]  (\ta-1,\ta) -- (\ta/\lamdba, \ta) -- (\ymax/\lamdba, \ymax) -- (50/12, 62/12); %((\ymax*(1+\lamdba)-1)/2, (\ymax*(1+\lamdba)+1)/2)
	\draw [decorate,decoration={brace,mirror,amplitude=10, raise=3}] (50/12, 62/12) -- (\ymax/\lamdba, \ymax) node [midway, xshift=32, yshift=18] {$\FF(A, 1-\lambda t)$};
	\draw [decorate,decoration={brace,amplitude=10, raise=3}] (\ta/\lamdba, \ta)  -- (\ymax/\lamdba, \ymax) node [midway, xshift=-40, yshift=2] {$ \EE_{\lambda}((1-\lambda) t)$};
	\draw [decorate,decoration={brace, mirror, amplitude=10, raise=3}] (\ta/\lamdba, \ta)  -- (\ta-1, \ta) node [midway, yshift=-24, xshift=18] {$\scriptsize \E_{\la t}^{t-A}(\vec \psi(t))$};
	\draw [thick] (\xmax - 1.5,\xmax - 0.5) node [above] {$u = A$} -- (\ta-1,\ta)  -- (\ta/\lamdba, \ta) -- (\ymax/\lamdba, \ymax); %node [below right] {$\xi =2t-A$};
	\draw (\ta-1,\ta+2) node {$\Om$};
\end{tikzpicture}
\caption{\label{fig:5} The quadrangle $\Om$ over which the energy identity is integrated is the gray region above.}
\end{figure}
%--------------------------------------------------------figure------------------------------------------------------------------------------------%
Letting $\bar v \to \infty$ above and recalling that $\FF(u, u_1)$ is increasing in $u_1$ we have 
\begin{align*} 
\E_{\la t}^{t-A}( \vec \psi(t))  &\le  \EE_{\la}( (1-\la)t) + \FF(A, (1-\la)t)\\
& \le  \EE_{\la}( (1-\la)t)+ \FF(A).
\end{align*} 
The proposition now follows from \eqref{e la to 0} and \eqref{ff to 0}.
\end{proof}

We will also need the following corollaries of Proposition~\ref{ctzcor1}:
\begin{cor}\label{ces lem}
Let $\vec \psi(t) \in \HH$ be a global wave map. Then 
\begin{align}\label{ces}
\limsup_{T \to \infty} \frac{1}{T} \int_A^T \int_0^{t-A} \psi_t^2(t, r) \, r\, dr\, dt \to 0 \quad \textrm{as} \quad A \to \infty.
\end{align}
\end{cor}
\begin{proof} We will use the following virial identity for solutions to \eqref{cp}: 
\begin{align}\label{vir}
\p_t( r^2 m) - \p_r(r^2\psi_t^2 + r^2 \psi_r^2 -\sin^2 \psi) + 2r \psi_t^2 = 0.
\end{align}
Now, fix $\de>0$ so that $\de<1/3$ and find $A_0, T_0$ so that for all $A \ge A_0$ and $t \ge T_0$ we have 
\begin{align*}
\E_{\de t}^{t-A}( \vec \psi(t)) \le \de.
\end{align*}
Then,
\begin{align*}
\int_0^{\de t} e(t, r) r^2 \, dr \le \E(\vec \psi(t)) \de t
\end{align*}
and as long as we ensure that $A \le 1/3 t$, we obtain
\begin{align*}
\int_{\de t}^{2t/3} e(t, r) \, r^2 \, dr \le \de t.
\end{align*}
This implies that 
\begin{align*}
\int_{0}^{2t/3} e(t, r) \, r^2 \, dr \le C\de t, \quad \textrm{and} \quad \int_{0}^{2t/3} e(t, r) \, r^3 \, dr \le C \de t^2.
\end{align*}
Let $\chi: \R \to [0,1]$ be a smooth cut-off function such that $\chi(x) = 1$ for $\abs{x} \le 1/3$, $\chi(x)=0$ for $\abs{x} \ge 2/3$ and $\chi^{\prime}(x) \le 0$. Then, using the virial identity \eqref{vir} we have
\begin{align*}
\frac{d}{dt} \int_0^{\infty} m(t, r) \chi(r/t) \, r^2 \, dr&= \int_0^{\infty} \p_t(r^2 m(t, r)) \, \chi(r/t) \, dr -\frac{2}{t^2} \int^{\I}_0 \psi_t \psi_r \, r^3 \chi^{\prime}(r/t) \, dr\\
&= \int_0^{\infty} \p_r( r^2( \psi_t^2 + \psi_r^2)- \sin^2(\psi)) \chi(r/ t) \, dr \\
& \quad - 2 \int_0^{\infty} \psi_t^2(t, r)\chi(r/t) \, r\, dr + O(\de)\\
&=  \frac{1}{t^2}\int_0^{\infty} ( r^2( \psi_t^2 + \psi_r^2)- \sin^2(\psi)) \chi^{\prime}(r/t) \,r \,  dr\\
&\quad - 2 \int_0^{\infty} \psi_t^2(t, r)\chi(r/t) \, r\, dr + O(\de)\\
&= - 2 \int_0^{\infty} \psi_t^2(t, r)\chi(r/t) \, r\, dr + O(\de).
\end{align*}
Integrating in $t$ between $0 $ and $T$ yields 
\begin{align*}
\int_0^T \int_0^{\infty} \psi_t^2(t, r) \chi(r/t) \, r \, dr \,dt \le C\de T
\end{align*}
with an absolute constant $C>0$. By the definition of $\chi(x)$ this implies 
\begin{align*}
\int_0^T \int_0^{t/3} \psi_t^2(t, r) \, r \, dr \,dt \le C\de T.
\end{align*}
Next, note that we have 
\begin{align*}
\int_A^T \int_{t/3}^{t-A} \psi_t^2(t, r)  \, r \,dr\, dt &\le  \int_A^{T_0} \E(\vec \psi) \, dt + \int_{T_0}^T \int_{t/3}^{t-A} e(t, r) \, r \, dr \, dt\\
&\le (T_0-A) \E(\vec \psi) + (T-T_0) \de.
\end{align*}
Therefore, 
\begin{align*} 
\frac{1}{T} \int_A^T \int_0^{t-A} \psi_t^2 (t, r) \, r \, dr \, dt \le C \de + \frac{T_0}{T} \E(\vec \psi)
\end{align*}
Hence, 
\begin{align*}
\limsup_{T \to \infty} \frac{1}{T} \int_A^T \int_0^{t-A} \psi_t^2 (t, r) \, r \, dr \, dt \le C \de 
\end{align*}
for all $A \ge A_0$, which proves \eqref{ces}.  
\end{proof}

\begin{cor} \label{psi to pi la} Let $\vec \psi(t) \in \HH$ be a smooth global wave map. Recall that $\vec \psi(t) \in \HH$ implies that there exists $k \in \Z$ such that for all $t$ we have $\psi(t, \infty)=  k \pi$. Then for any $\la>0$ we have 
\begin{align}
\|\psi(t)-  \psi(t, \infty)\|_{L^{\infty}(r \ge \la t)} \to 0 \quad \textrm{as} \quad t \to \infty.
\end{align}
\end{cor}
Before proving Corollary~\ref{psi to pi la}, we can combine Proposition~\ref{ctzcor1} and Corollary~\ref{psi to pi la} to immediately deduce the following result.  
\begin{cor} \label{HxL2 for} Let $\vec \psi(t) \in \HH$ be a global wave map. Let $0 < \la<1$. Then we have 
\begin{align}\label{limsup H}
\limsup_{t \to \infty}\|\vec \psi(t)- ( \psi(t, \infty), 0)\|_{H \times L^2( \la t \le r \le t-A)}^2 \to 0\quad \textrm{as} \quad A \to \infty.
\end{align}
\end{cor}
\begin{proof} Say $\vec \psi(t) \in \HH_k$. Observe that  Corollary~\ref{psi to pi la} shows that for $t_0$ large enough we have, say, $$\abs{\psi(t, r)- k \pi} \le \frac{\pi}{100}$$  for all $t \ge t_0$ and $r \ge \la t$. This in turn   implies that for $t \ge t_0$ we can find a $C >0$ such that 
\begin{align*}
\abs{ \psi(t, r)-k\pi}^2 \le C \sin^2( \psi(t, r)) \quad \forall\, t \ge t_0, \, r \ge \la t.
\end{align*}
Now \eqref{limsup H} follows directly from \eqref{limsup la}. 
\end{proof}

The first step in the proof of Corollary~\ref{psi to pi la} is the following lemma:
\begin{lem}\label{psi-pi A lem}
Let $\vec \psi(t) \in \HH$ be a smooth global wave map. Let $R>0$ and suppose that the initial data $\vec \psi(0)= ( \psi_0, \psi_1) \in \HH_1$ satisfies $\supp(\p_r \psi_0),  \supp( \psi_1) \subset B(0, R)$. Then for any $t \ge 0$ and for any $A <t$  we have
\begin{align}\label{psi-pi A}
\| \psi(t)- \psi(t, \infty) \|_{L^{\infty}( r \ge t-A)} \le \sqrt{\E( \vec \psi)} \sqrt{\frac{A+R}{t-A}}.
\end{align}
\end{lem}

\begin{proof} By the finite speed of propagation we note that for each $t \ge 0 $ we have $ \supp( \psi_r(t)) \subset B(0, R+t)$. Hence, for all $t \ge 0$ we have
\begin{align*}
\abs{\psi(t, r) - \psi(t, \infty)} &\le  \int_r^{\infty} \abs{\psi_r(t, r')} \, dr' \\
&\le  \left(\int_{r}^{R+t} \psi_r^2(t, r') \, r' \, dr'\right)^{\frac{1}{2}} \left( \int_r^{R+t} \frac{1}{r'} \, dr' \right)^{\frac{1}{2}}\\
&\le \sqrt{\E( \vec \psi)} \sqrt{ \log\left( \frac{t+R}{r}\right)}.
\end{align*}
Next observe that if $r \ge t-A$ then 
\begin{align*}
 \log\left( \frac{t+R}{r}\right) \le \log\left(1+  \frac{A+R}{r}\right) \le \log\left( 1+ \frac{A+R}{t-A}\right) \le \frac{A+R}{t-A}.
 \end{align*}
 This proves \eqref{psi-pi A}. 
\end{proof}

\begin{proof}[Proof of Corollary~\ref{psi to pi la}] Say $\psi(t) \in \HH_k$, that is $\psi(t, \infty)= k\pi$ for all $t$. First observe that by an approximation argument, it suffices to consider wave maps $\vec \psi(t) \in \HH_k$ with initial data $\vec \psi(0) = ( \psi_0, \psi_1) \in \HH_k$ with $$\supp(\p_r \psi_0), \supp(\psi_1) \subset B(0, R)$$ for $R>0$ arbitrary, but fixed.  Now, let $t_n \to \infty$ be any sequence and set $$A_n:= \sqrt{t_n}.$$ Then, for each $r \ge \la t_n$ we have 
\begin{align*} 
\abs{\psi(t_n, r) - k\pi} \le \|\psi(t_n)- k\pi\|_{L^{\infty}( \la t_n \le r \le t_n- A_n)} + \|\psi(t_n)- k \pi\|_{L^{\infty}(  r \ge t_n- A_n)}.
\end{align*}
By Lemma~\ref{psi-pi A lem} we know that 
\begin{align}\label{>An}
\|\psi(t_n)- k\pi\|_{L^{\infty}(  r \ge t_n- A_n)} \le \sqrt{\E(\psi)} \sqrt{ \frac{ \sqrt{t_n}+R}{t_n- \sqrt{t_n}}} \to 0 \quad \textrm{as} \quad n  \to \infty.
\end{align}
Hence it suffices to show that 
\begin{align*}
\|\psi(t_n)- k\pi\|_{L^{\infty}( \la t_n \le r \le t_n- A_n)} \to 0 \quad \textrm{as} \quad n \to \infty.
\end{align*}
To see this, first observe that \eqref{>An} implies that $$\psi(t_n, t_n-A_n) \to k\pi$$ as $n \to \infty$. Therefore it is enough to show that 
\begin{align}\label{enough}
\|\psi(t_n)- \psi(t_n, t_n-A_n)\|_{L^{\infty}( \la t_n \le r \le t_n- A_n)} \to 0 \quad \textrm{as} \quad n \to \infty.
\end{align}
With $G$ defined as in \eqref{G def} we can combine \eqref{b R} and Proposition~\ref{ctzcor1} to deduce that for all $r \ge \la t_n$ we have
\begin{align*}
\abs{G(\psi(t_n, r))- G( \psi(t_n, t_n-A_n))} \le \frac{1}{2} \E_{\la t_n}^{t_n -A_n}( \vec \psi(t_n))\to 0.
\end{align*}
as $ n \to \infty$. This immediately implies \eqref{enough} since $G$ is a continuous,  increasing function. 
 \end{proof}

\section{Profiles for global degree one solutions with energy below $3\E(Q)$} 

In this section we carry out the proof of Theorem~\ref{glob sol} $(2)$. We start by first deducing the conclusions along a sequence of times. To be specific, we establish the following proposition: 

\begin{prop}\label{global} Let $\psi(t) \in \HH_1$ be a global solution to \eqref{cp} with $$\E(\vec \psi) = \E(Q) + \eta< 3 \E(Q).$$ Then there exist a sequence of times $\tau_n \to \infty$, a sequence of scales $\la_n\ll \tau_n$, a  solution $\vec \fy_L(t) \in \HH_0$ to the linear wave equation \eqref{2d lin}, %with $\E(\vec \fy_L)= \eta$, 
and a decomposition  
\begin{align}\label{dec seq}
\vec \psi(\tau_n) =   \vec \fy_L(\tau_n) + \left(Q\left(\cdot/\la_n\right), 0\right) + \vec\epsilon(\tau_n)
\end{align}
such that $\vec \epsilon(\tau_n) \in \HH_0$ and $\vec \epsilon(\tau_n)\to 0$  in $H \times L^2$ as $n \to \infty$.
\end{prop}

To prove Proposition~\ref{global} we proceed in several steps. We first construct the sequences $\tau_n$  and $\la_n$ while identifying the large profile, $Q( \cdot/ \la_n)$. Once we have done this, we extract the radiation term $\fy_L$. In the last step, we prove strong convergence of the error $$ \vec \epsilon(\tau_n):= \vec \psi(\tau_n) -   \vec \fy_L(\tau_n) - \left(Q\left(\cdot/\la_n\right), 0\right) \to 0$$ in the space $H \times L^2$. 

%In what follows we fix $\vec \psi(t) \in \HH_1$ a global in time wave map with energy $\E(\vec \psi

 \subsection{The harmonic map at $t=+ \infty$}
 
 Here we prove the analog of Struwe's result \cite[Theorem $2.1$]{St}    for global wave maps of degree different than zero, i.e.,  $\psi(t) \in \HH \backslash \HH_0$ for all $t \in [0, \infty)$.  This will allow us to identify the sequences $\tau_n$, $\la_n$ and the harmonic maps $Q( \cdot/ \la_n)$ in the decomposition \eqref{dec seq}.   %In particular, we prove the following result: %We show that for each global wave map $\psi(t) \in \HH_1$, we can produce a sequence of times $t_n \to \infty$, scales $\mu_n\ll t_n$, and a corresponding sequence of rescaled wave maps $$\vec \psi_n(t, r):=\left( \psi(t_n + \la_n t, \la_n r),  \la_n \dot \psi(t_n + \la_n t, \la_n r)\right)$$ that converge locally to $\pm Q_{\la_0}$ in the space-time norm $L^2_t((-1, 1); H^1 \times L^2)_{\loc}$ as $n \to \infty$. 
 
\begin{thm} \label{struwe} Let  $\vec \psi(t) \in \HH \backslash \HH_0$ be a smooth, global solution to \eqref{cp}.  Then, there exists a sequence of times $t_n \to \infty$ and a sequence of scales $\la_n \ll t_n$ so that the following results hold: Let 
  \begin{align}
  \vec \psi_n(t, r):=\left( \psi(t_n + \la_n t, \la_n r),  \la_n \dot \psi(t_n + \la_n t, \la_n r)\right)
  \end{align}  
  be the global wave map evolutions associated to the initial data $$\vec \psi_n(r):= (\psi(t_n, \la_n r), \la_n \dot \psi(t_n, \la_n, r)).$$
  Then, there exists $\la_0>0$ so that 
  \begin{align*}
  \vec \psi_n \to (\pm Q(\cdot/ \la_0), 0) \quad \textrm{in} \quad L^2_t([0, 1) ; H^1\times L^2 )_{\loc}.
  \end{align*}
  
\end{thm}

We begin with the following lemma, which  follows from Corollary~\ref{ces lem} and is the global-in-time version of \cite[Corollary $2.9$]{CKLS1}. The statement and proof are also very similar to \cite[Lemma $4.4$]{DKM3} and \cite[Corollary $5.3$]{DKM1}. 

\begin{lem} \label{sig lem} Let $\vec \psi(t) \in \HH$ be a smooth global wave map. Let $A:(0, \infty) \to (0, \infty)$ be any increasing function such that $A(t) \nearrow \infty$ as $t  \to \infty$ and $A(t)  \le t$ for all  $t$. Then, there exists a sequence of times $t_n \to \infty$ such that 
\begin{align}\label{sig}
\lim_{n \to \infty} \sup_{\s>0} \frac{1}{\s} \int_{t_n}^{t_n+\s} \int_0^{t- A(t_n)} \dot \psi^2(t, r) \, r \, dr \, dt = 0.
\end{align}
\end{lem}
 
 \begin{proof} The proof is analogous to the argument given in \cite[Corollary $5.3$]{DKM1}. We argue by contradiction. The existence of a sequence of times $t_n$ satisfying \eqref{sig} is equivalent to the statement
\begin{align*} 
\forall A(t) \nearrow \infty \, \, \textrm{with} \, \,& A(t) \le t \, \,  \textrm{as} \, \, t \to \infty, \, \, \forall \de>0, \, \, \forall T_0 >0, \, \, \exists \tau  \ge T_0 \, \, \textrm{so that} \, \, \\
&\sup_{\s>0} \frac{1}{ \s} \int_{\tau}^{\tau+\s} \int_0^{t- A(\tau)} \dot \psi^2(t, r) \, r \, dr \, dt \le \de.
\end{align*} 
 So we assume that \eqref{sig} fails. Then, 
 \begin{align}\notag
 \exists A(t) \nearrow \infty \, \, \textrm{with} \, \,& A(t) \le t \, \,  \textrm{as} \, \, t \to \infty, \, \, \exists \de>0, \, \, \exists T_0 >0, \, \, \,   \forall \tau  \ge T_0,  \, \, \exists \s>0 \,\, \textrm{so that} \\
& \frac{1}{ \s} \int_{\tau}^{\tau+\s} \int_0^{t- A(\tau)} \dot \psi^2(t, r) \, r \, dr \, dt > \de \label{sig big}.
\end{align} 
 Now, by Corollary~\ref{ces lem} we can find a large $A_1$ and a $T_1=T_1(A_1)>T_0$ so that for all $T \ge T_1$ we have
 \begin{align} \label{ces small1}
 \frac{1}{T} \int_{A_1}^T \int_0^{t-A_1} \dot \psi^2(t, r) \, r\, dr \, dt  \le \de/100.
\end{align}
 Since $A(t) \nearrow \infty$  we can fix $T>T_1$ large enough so that $A(t) \ge A_1$ for all $t \ge T$. Define the set $X$ as follows: 
 \begin{align*}
 X:=\left\{ \s>0 \, : \, \frac{1}{\s} \int_{T}^{T+\s} \int_0^{t- A(T)} \dot \psi^2(t, r) \, r \, dr \, dt \ge\de\right\}.
 \end{align*}
 Then $X$ is nonempty by \eqref{sig big}. Define $\rho:= \sup X$. We claim that $\rho  \le T$. To see this assume that there exists $\s \in X$ so that $\s \ge T$. Then we would have 
 \begin{align*}
  T+ \s  \le 2 \s.
 \end{align*}
 This in turn implies, using \eqref{ces small1}, that
 \begin{align*}
 \frac{1}{2 \s} \int_T^{T+ \s} \int_0^{t-A(T)} \dot \psi^2(t, r) \, r \, dr \, dt  \le \frac{1}{ T+ \s} \int_{A_1}^{T+ \s} \int_{0}^{t- A_1}\dot \psi^2(t, r) \, r \, dr \, dt  \le \de/ 100
 \end{align*}
 where we have also used the  fact that $A(T) \ge A_1$. This would mean that 
 \begin{align*}
  \frac{1}{ \s} \int_T^{T+ \s} \int_0^{t-A(T)} \dot \psi^2(t, r) \, r \, dr \, dt   \le \de/50,
  \end{align*} 
  which is impossible since we assumed that $\s \in X$. Therefore $\rho \le T$. Moreover, we know that 
  \begin{align}\label{1int}
  \int_T^{T+ \rho} \int_0^{T- A(T)}  \dot \psi^2(t, r) \, r \, dr \, dt \ge \de \rho.
  \end{align}
  Now, since $T+ \rho >T>T_1>T_0$ we know that there exists $\s>0$ so that 
  \begin{align*}
  \int_{T+ \rho}^{T+ \rho + \s} \int_0^{t-A(T+ \rho)}  \dot \psi^2(t, r) \, r \, dr \, dt > \de \s.
  \end{align*}
  Since $A(t)$ is increasing, we have $A(T) \le A(T+ \rho)$ and hence the above implies that 
  \begin{align} \label{2int}
    \int_{T+ \rho}^{T+ \rho + \s} \int_0^{t-A(T)} \dot \psi^2(t, r) \, r \, dr \, dt > \de \s.
    \end{align}
    Summing \eqref{1int} and \eqref{2int} we get
    \begin{align*}
    \int_{T}^{T+ \rho + \s} \int_{0}^{t- A(T)} \dot \psi^2(t, r) \, r \, dr \, dt > \de (\s+ \rho),
    \end{align*}
    which means that $\rho + \s \in X$. But this contradicts that fact that $\rho= \sup X$. 
 %there exists $A(t) \to  \infty$ with $A(t) \ll t$ as $t \to \infty$ and there exists $\de>0$ and there exists 
\end{proof}

 The rest of the proof of Theorem~\ref{struwe} will follow the same general outline of  \cite[proof of Theorem $2.1$]{St}. Let $\vec \psi(t) \in \HH_1$ be a smooth global wave map. 
 
 We begin by choosing a scaling parameter. Let $\de_0>0$ be a small number, for example $\de_0=1$ would work. For each $t \in (0, \infty)$ choose $\la(t)$ so that 
 \begin{align}\label{mu def}
  \de_0 \le \E_0^{2 \la(t)}( \vec \psi(t)) \le2 \de_0.
  \end{align}
  Then using the monotonicity of the energy on interior cones we know that for each $ \abs{\tau} \le  \la(t)$ we have 
  \begin{align}\label{up bound}
  \E_0^{\la(t)} ( \vec \psi(t+ \tau)) \le \E_0^{2\la(t)- \abs{\tau}}( \vec \psi(t+ \tau)) \le \E_0^{2 \la(t)}( \vec \psi(t)) \le 2 \de_0.
  \end{align}
  Similarly, we have 
 \begin{align}\label{low bound}
 \de_0 \le \E_0^{2 \la(t) + \abs{\tau}}( \vec \psi(t+ \tau)) \le \E_0^{3\la(t)}( \vec \psi(t+ \tau)).
 \end{align}
\begin{lem} 
Let $\vec \psi(t) \in\HH \backslash \HH_0$ and $\la(t)$ be defined as above. Then we have $\la(t) \ll t$ as $t \to \infty$. 
\end{lem}
\begin{proof} 
Suppose $\vec \psi \in \HH_k$ for $k \ge 1$. It suffices to show that for all $\la>0$ we have $\la(t) \le \la t $ for all $t$ large enough. Fix $\la>0$. By Corollary~\ref{psi to pi la} we have 
\begin{align}\label{psi pi l}
\| \psi(t)-  k \pi\|_{L^{\infty}(r \ge \la t)} \to 0 
\end{align}
as $t \to \infty$. For the sake of finding a contradiction, suppose that there exists a sequence $t_n \to \infty$  with 
$\la(t_n) \ge \la t_n$  for all $n \in \N$. By \eqref{b R} and \eqref{psi pi l} we would then have that 
\begin{align*}
\E_0^{2 \la(t_n)}( \vec \psi(t_n)) \ge \E_0^{\la t_n}( \vec \psi(t_n)) \ge 2 G( \psi(t_n, \la t_n)) \to 2 G(k \pi) \ge 4 > 2\de_0,
\end{align*}
which contradicts \eqref{mu def} as long as we ensure that $\de_0<2$. 
\end{proof}

We can now complete the proof of Theorem~\ref{struwe}. 
\begin{proof}[Proof of Theorem~\ref{struwe}]
Let  $\la(t)$ be defined as in \eqref{mu def}.  Choose another scaling parameter $A(t)$ so that $A(t) \to \infty$ and $\la(t) \le A(t) \ll t$ for  $t \to \infty $, for example one could take $A(t):=\max\{ \ti \la(t), t^{1/2}\}$ where $\ti \la(t) := \sup_{0  \le s  \le t }\la(s)$.  %(although we will not need the assumption $A(t) \ll t$ in this argument we will use it later in the proof of Lemma~\ref{tech lem}). 
By Lemma~\ref{sig lem} we can find a sequence $t_n \to \infty$ so that by setting $\la_n:=\la(t_n)$ and $A_n:=A(t_n)$ we have
 \begin{align*}
 \lim_{n \to \infty} \frac{1}{\la_n}  \int_{t_n}^{t_n + \la_n} \int_0^{t-A_n} \dot \psi^2(t, r) \, r\, dr \, dt  = 0.
\end{align*}
%where $A_n:=A( t_n) \to \infty$ is chosen as in Lemma~\ref{t av dec lem} so that $A_n \to \infty$ and $\la_n \le A_n \ll t_n$. %defined as follows: 
 %\begin{align*} 
 %A_n:= \begin{cases} \la_n \quad \textrm{if} \quad \la_n \to \infty\\  \sqrt{t_n}\quad \textrm{if} \quad  0<\la_n \le C < \infty.
% \end{cases}
 %\end{align*}
Now define a sequence of global wave maps $\vec \psi_n(t) \in \HH \backslash \HH_0$ by
\begin{align*}
\vec \psi_n(t, r):= \left( \psi(t_n + \la_n t, \la_n r),  \la_n \dot \psi(t_n + \la_n t, \la_n r)\right).
  \end{align*}  
and write the full wave maps in coordinates on $\Sp^2$ as $U_n(t, r, \om):= ( \psi_n(t, r), \om)$. Observe that we have 
\begin{align}\label{dt to 0}
\int_{0}^1 \int_0^{r_n} \dot\psi_n^2(t, r) \, r\, dr \, dt \to 0 \quad \textrm{as} \quad n \to \infty
\end{align}
where $r_n:= (t_n-A_n)/ \la_n  \to \infty$ as $n \to \infty$ by our choice of $A_n$. 
Also note that
\begin{align*}
\E( \vec \psi_n(t))= \E( \vec \psi(t_n+\la_n t))= \E( \vec\psi)= C.
\end{align*}
This implies that the sequence $\vec \psi_n$ is uniformly bounded in $L^{\infty}_t(\dot{H}^1 \times L^2)$. Note that \eqref{b R} implies that $\psi_n$ is uniformly bounded in $L^{\infty}_tL^{\infty}_x$. Hence we can extract a further subsequence so that 
\begin{align*}
&\vec \psi_n  \rightharpoonup \vec{\psi}_{\infty} \quad \textrm{weakly in} \quad L^2_t (H^1 \times L^2)_{\loc}
%\textrm{or} \quad &U_n \rightharpoonup  U_{\infty} \quad \textrm{weakly in} \quad H^1( [0, \infty) \times \R^2; \Sp^2)
\end{align*}
and, in fact, locally uniformly on $[0, 1) \times (0, \infty)$. 
By \eqref{dt to 0}, the limit $$ \vec\psi_{\infty}(t, r) = (  \psi_{\infty}(r), 0) \quad \forall (t, r)  \in [0, 1) \times (0, \infty)$$  and is thus a time-independent weak solution to \eqref{cp} on $[0, 1) \times (0, \infty)$. This means that the corresponding full, weak wave map $\ti U_{\infty}(t, r, \om)=  U_{\infty}(r, \om):=( \psi_{\infty}(r), \om)$ is a time-independent  weak solution to \eqref{cp i} on $[0, 1) \times \R^2 \,\backslash\,\{0\}$. By H\'{e}lein's theorem \cite[Theorem $2$]{Hel}, $$U_{\infty}: \R^2\, \backslash\, \{0\} \to \Sp^2$$ 
is a smooth finite energy, co-rotational   harmonic map. By Sacks-Uhlenbeck, \cite{SU}, we can then extend $ U_{\infty}$ to a smooth finite energy, co-rotational harmonic map $U: \R^2 \to \Sp^2$. Writing $U(r, \om) = ( \psi_{\I}(r), \om)$, we have either $\psi_{\infty} \equiv 0$ or $\psi_{\infty}=  \pm Q(\cdot/ \la_0)$ for some $\la_0>0$.

Following Struwe, we can also establish strong local convergence 
\begin{align}\label{strong loc}
\vec \psi_n \to (\psi_{\infty}, 0) \quad \textrm{in} \quad L^2_t([0, 1) ; H^1\times L^2 )_{\loc}
\end{align}
using the equation \eqref{cp i} and the local energy constraints from \eqref{up bound}:
\begin{align*}
\E_0^{1}( \vec \psi_n(t)) \le 2 \de_0, \quad \E_0^1( \psi_{\infty}) \le 2 \de_0,
\end{align*}
which hold uniformly in $n$ for $\abs{t} \le 1$. For the details of  this argument we refer the reader to \cite[Proof of Theorem $2.1$ (ii)]{St}. Finally we note that the strong local convergence in \eqref{strong loc} shows that indeed $\psi_{\infty} \not \equiv 0$ since by \eqref{low bound} we have 
\begin{align*}
\de_0 \le \E_0^{3}( \vec \psi_n(t))
\end{align*}
uniformly in $n$ for each $ \abs{t} \le 1$. Therefore we can conclude that there exists $\la_0>0$ so that $\psi_{\infty}(r)= \pm Q(r/ \la_0)$. 
\end{proof}

As in \cite{CKLS1}, the following  consequences of Theorem~\ref{struwe}, which hold for global degree one wave maps with energy below $3\E(Q)$, will be essential in what follows. 

\begin{cor}  \label{str 3eq}Let $\psi(t) \in \HH_1$ be a smooth global wave map  such that $\E(\vec \psi)<3 \E(Q)$. %Recall that $\vec \psi(t) \in \HH_1$ means that $\psi(t, 0)=0, \psi(t,\infty)= \pi$. 
Then we have
\begin{align}\label{Hloc}
\psi_n - Q( \cdot /\la_0)  \to 0\quad{as} \quad n \to \infty \quad \textrm{in} \quad L^2_t([0, 1); H)_{\loc},
\end{align}
with $\psi_n(t, r)$, $\{t_n\}$,  $\{\la_n\}$, and $\la_0$  as in Theorem~\ref{struwe}. %In addition, there exists another sequence of times $\bar t_n \to \infty$ and a sequence of points $r_n \in (0, \bar t_n)$ such that 
%\begin{align} \label{pi seq}
% \psi(\bar t_n, r_n) \to  \pi \quad \textrm{as} \quad n \to \infty
%\end{align} 
\end{cor} 
Corollary~\ref{str 3eq} is the global-in-time analog of \cite[Corollary $2.13$]{CKLS1}. For the details, we refer the reader to \cite[Proof of Lemma $2.11$, Lemma $2.12$,  and  Corollary $2.13$]{CKLS1}. At this point we note that we can, after a suitable rescaling, assume, without loss of generality, that $\la_0$ in Theorem~\ref{struwe}, and Corollary~\ref{str 3eq}, satisfies $\la_0 =1$.

Arguing as in \cite[Proof of Proposition $5.4$]{CKLS1} we can also deduce the following consequence of Theorem~\ref{struwe}.

\begin{prop} \label{str diag} Let $\psi(t) \in \HH_1$ be a smooth global wave map such that $\E(\vec \psi)<3 \E(Q)$. %Let $\al_n$ be any sequence such that $\al_n \to \infty$. 
Then, there exists a sequence $\al_n \to \infty$, a sequence of times $\tau_n \to \infty$, and  a sequence of scales $\la_n \ll \tau_n$ with $\al_n \la_n \ll \tau_n$, so that 
\begin{itemize} 
\item[($a$)] As $n \to \infty$ we have 
\begin{align} \label{psi dot}
\lim_{n \to \infty}\int_0^{\tau_n - A_n} \dot{\psi}^2(\tau_n, r) \, r\, dr  \to 0,
\end{align}
where $A_n \to \infty$  satisfies $\la_n \le A_n \ll \tau_n$.  
\item[($b$)] As $n \to \infty$  we have
\begin{align} \label{psi-Q}
\lim_{n  \to \infty}\int_0^{\al_n \la_n} \left(\abs{\psi_r(\tau_n, r) - \frac{Q_r(r/ \la_n)}{\la_n}}^2 + \frac{\abs{\psi(\tau_n, r)- Q(r/ \la_n)}^2}{r^2}\right) \, r\, dr=0.
\end{align}
\end{itemize}
\end{prop}

\begin{rem} \label{rem str}
Proposition~\ref{str diag} follows directly from Lemma~\ref{sig lem}, Corollary~\ref{str 3eq} and a diagonalization argument. As mentioned above, we refer the reader to \cite[Proposition $5.4$ $(a)$, $(b)$]{CKLS1} for the details.  Also note that $\tau_n \in [t_n, t_n + \la_n]$ where $t_n \to \infty$ is the sequence in Proposition~\ref{str diag}. Finally $A_n:=A(t_n)$ is the sequence that appears in the proof of Theorem~\ref{struwe}. 
\end{rem}

As in \cite{CKLS1} we will also need the following simple consequence of Proposition~\ref{str diag}. 
\begin{cor}\label{tech} Let $\al_n, \la_n, $ and $ \tau_n$ be defined as in Proposition~\ref{str diag}.  Let $\be_n \to \infty$ be any sequence such that $\be_n< c_0 \al_n$ for some $c_0<1$. Then, for every $0<c_1<C_2$ such that $C_2c_0<1$ there exists $\ti \be_n$ with $c_1  \be_n \le\ti \be_n \le C_2 \be_n$ such that 
\begin{align}
\psi( \tau_n, \ti \be_n \la_n) \to \pi \quad \textrm{as} \quad n \to \infty.
\end{align}
\end{cor}

\subsection{Extraction of the radiation term}

In this subsection we construct what we will refer to as the radiation term, $\fy_L(t) \in \HH_0$ in the decomposition \eqref{dec seq}. 
\begin{prop}\label{rad prop} Let $\psi(t) \in \HH_1$ be a global wave map with $\E( \vec \psi)= \E(Q) + \eta < 3 \E(Q)$. Then there exists a solution $\fy_L(t) \in \HH_0$ to the linear wave equation \eqref{2d lin} so that for all $A \ge 0$ we have 
\begin{align}\label{rad}
\| \vec \psi(t) - (\pi, 0) - \vec \fy_L(t)\|_{H \times L^2( r \ge t-A)} \to 0 \quad \textrm{as} \quad t \to \infty.
\end{align}
Moreover,  for $n$ large enough we have 
\begin{align}\label{fy 2eq}
\E( \vec \fy_{L}(\tau_n)) \le C < 2 \E(Q).
\end{align}
\end{prop}

\begin{proof} 
To begin we pick any $\al_n \to \infty$ and find $\tau_n, \la_n$ as in Proposition~\ref{str diag}. Now let $\be_n \to \infty$ be any other sequence such that $\be_n \ll \al_n$. By Corollary~\ref{tech} we can assume that 
\begin{align}\label{to pi}
\psi(\tau_n, \be_n \la_n) \to \pi
\end{align}
as $n \to \infty$. We make the following definition: 
\begin{align} 
&\phi_n^0(r) = \begin{cases} \pi - \frac{\pi-\psi(\tau_n, \be_n\la_n)}{\be_n \la_n} r \quad \textrm{if} \quad 0\le r\le \be_n \la_n\\ \psi(\tau_n, r) \quad \textrm{if} \quad \be_n \la_n \le r < \infty \end{cases} \\
&\phi_n^1(r) = \begin{cases}0 \quad \textrm{if} \quad 0\le r\le \be_n \la_n\\ \dot{\psi}(\tau_n, r) \quad \textrm{if} \quad \be_n \la_n\le r < \infty. \end{cases} 
\end{align}
We claim that $\vec \phi_n:=( \phi_n^0, \phi_0^1) \in \HH_{1, 1}$ and $\E( \vec \phi_n) \le C < 2\E(Q)$.  Clearly $\phi_n^0(0)= \pi$ and $\phi_n^0(\infty)= \pi$. We claim that 
\begin{align} \label{E phi bela}
\E_{\be_n \la_n}^{\infty}(\vec \phi_n) = \E_{\be_n \la_n}^{\infty}(\vec \psi(\tau_n)) \le \eta + o_n(1).
\end{align} 
Indeed, since $\psi(\tau_n, \be_n \la_n) \to \pi$ we have $G(\psi(\tau_n, \be_n \la_n)) \to 2= \frac{1}{2}\E(Q)$ as $n \to \infty$. Therefore, by \eqref{b R} we have $$\E_0^{\be_n \la_n}(\psi(\tau_n), 0)  \ge 2G(\psi(\tau_n , \be_n \la_n)) \ge \E(Q)- o_n(1)$$ for large $n$ which  proves \eqref{E phi bela} since $ \E_{\be_n \la_n}^{\infty}(\vec \psi(\tau_n)) =  \E_{0}^{\infty}(\vec \psi(\tau_n))-  \E_{0}^{\be_n \la_n}(\vec \psi(\tau_n))$. 

We can also directly compute $\E_{0}^{\be_n \la_n} (\phi_n^0,0)$. Indeed,
\begin{align*} 
\E_{0}^{\be_n \la_n} (\phi_n^0,0)&= \int_0^{\be_n \la_n} \left( \frac{\pi-\psi(\tau_n, \be_n \la_n)}{\be_n \la_n}\right)^2\, r\, dr + \int_0^{\be_n \la_n} \frac{\sin^2\left( \frac{\pi-\psi(\tau_n, \be_n \la_n)}{\be_n \la_n} r\right)}{r} \, dr\\ \\
& \le C \abs{\pi- \psi(\tau_n, \be_n \la_n)}^2  \to 0 \quad \textrm{as} \quad   n \to \infty.
\end{align*} 
Hence $\E( \vec \phi_n)\le \eta + o_n(1)$. This means that for large enough $n$ we have the uniform estimates $\E( \vec \phi_n) \le C < 2\E(Q)$. Therefore, by \cite[Theorem $1.1$]{CKLS1}, (which holds with exactly the same statement in $\HH_{1, 1}$ as in $\HH_0= \HH_{0, 0}$), we have that the wave map evolution $\vec \phi_n(t) \in \HH_{1,1}$ with initial data $\vec \phi_n$ is global in time and scatters  to $\pi$ as $t \to \pm \infty$. The scattering statement means that for each $n$ we can find initial data $\vec \phi_{n, L}$ so that the solution, $S(t)\vec \phi_{n, L}$, to the linear wave equation \eqref{2d lin} satisfies
\begin{align*}
\| \vec \phi_n(t)- ( \pi, 0)- S(t)\vec \phi_{n, L}\|_{H \times L^2} \to 0 \quad \textrm{as} \quad t \to \infty.
\end{align*} 
Abusing notation, we set $$\vec\phi_{n,L}(t):= S(t- \tau_n) \vec \phi_{n, L}.$$
By the definition of $\vec \phi_n$ and the finite speed of propagation observe that we have 
\begin{align*}
\phi_n(t-\tau_n, r) = \psi(t, r) \quad \forall r \ge t- \tau_n + \be_n \la_n.
\end{align*}
 Therefore, for all fixed $m$ we have 
\begin{align}\label{scat fix m}
\|\vec \psi(t)- (\pi, 0)- \vec \phi_{m, L}(t)\|_{H \times L^2(r \ge t- \tau_m + \be_m \la_m)} \to 0 \quad \textrm{as} \quad t \to \infty,
\end{align}
and, in particular 
\begin{align}\label{scat nm}
\|\vec \phi_n - (\pi, 0) - \vec \phi_{m, L}(\tau_n)\|_{H \times L^2(r \ge \tau_n- \tau_m + \be_m \la_m)} \to 0 \quad \textrm{as} \quad n \to \infty.
\end{align}
Now set $\vec \fy_n = (\fy_n^0, \fy_n^1):= ( \phi_n^0, \phi_n^1) -( \pi, 0) \in \HH_0$. We have $\E( \vec \fy_n) \le C < 2\E(Q)$ by construction. Therefore the sequence $S(- \tau_n) \vec \fy_n$ is uniformly bounded in $H \times L^2$. 
Let $\vec \fy_L= (\fy_L^0, \fy_L^1)\in \HH_0$  be the weak limit of $S(- \tau_n) \vec \fy_n$ in $H \times L^2$ as $n \to \infty$, i.e., 
\begin{align*} 
S(-\tau_n) \vec \fy_n \rightharpoonup \vec \fy \quad \textrm{weakly in } \quad H \times L^2
\end{align*}
as $n \to \infty$. Denote by $\vec \fy_L(t):=S(t) \vec \fy_L$ the linear evolution of $\vec \fy_L$ at time $t$.  Following the construction  in \cite[Main Theorem]{BG} %\cite[Corollary $2.15$]{CKLS1}
 we have the following profile decomposition for $\vec \fy_n$: 
\begin{align}\label{prof for fy}
\vec \fy_n(r)= \vec \fy_L(\tau_n, r) + \sum_{j=2}^k \left( \fy_L^j(t_n^j/\la_n^j, r/ \la_n^j), \, \frac{1}{\la_n^j} \dot \fy_L^j(t_n^j/\la_n^j, r/ \la_n^j)\right) + \vec \ga_n^k(r)
\end{align}
where if we label $\fy_L=: \fy_L^1$,  $\tau_n =: t_n^1$, and $\la_n^1 = 1$ this is exactly  a profile decomposition as in \cite[Corollary $2.15$]{CKLS1}.  Now observe that for each fixed $m$ we can write 
\begin{multline}\label{prof n-m}
\vec \fy_n(r)- \vec \phi_{m, L}(\tau_n, r) = \vec \fy_L(\tau_n, r)- \vec \phi_{m, L}(\tau_n, r) \\
\quad + \sum_{j=2}^k \left( \fy_L^j(t_n^j/\la_n^j, r/ \la_n^j), \, \frac{1}{\la_n^j} \dot \fy_L^j(t_n^j/\la_n^j, r/ \la_n^j)\right) + \vec \ga_n^k(r)
\end{multline}
and  \eqref{prof n-m} is still a profile decomposition in the sense of \cite[Corollary $2.15$]{CKLS1} for the sequence $\vec \fy_n(r)- \vec \phi_{m, L}(\tau_n, r)$. Since the pseudo-orthogonality of the $H\times L^2$ norm is preserved after sharp cut-offs, see \cite[Corollary $8$]{CKS} or \cite[Proposition $2.19$]{CKLS1}, we then have 
\begin{multline*}
\| \vec \fy_n- \vec \phi_{m, L}(\tau_n)\|_{H \times L^2 (r \ge \tau_n- \tau_m+ \be_m \la_m)}^2=  \|\vec \fy_L(\tau_n)- \vec \phi_{m, L}(\tau_n)\|_{H \times L^2 (r \ge \tau_n- \tau_m+ \be_m \la_m)}^2
 \\+ \sum_{j=2}^k \|\vec \fy_L^j(t_n^j/\la_n^j)\|_{H \times L^2 (r \ge \tau_n- \tau_m+ \be_m \la_m)} ^2
+\| \vec \ga_j^k\|_{H \times L^2 (r \ge \tau_n- \tau_m+ \be_m \la_m)}^2+ o_n(1)
\end{multline*}
Note that \eqref{scat nm} implies that the left-hand-side above tends to zero as $n \to \infty$. Therefore, since all of the terms on right-hand-side are non-negative we can deduce that 
\begin{align*}
\|\vec \fy_L(\tau_n)- \vec \phi_{m, L}(\tau_n)\|_{H \times L^2 (r \ge \tau_n- \tau_m+ \be_m \la_m)}^2 \to 0 \quad \textrm{as} \quad n \to \infty.
\end{align*}
Since, 
\begin{align*}
\vec \fy_L(\tau_n)- \vec \phi_{m, L}(\tau_n) = S(\tau_n)( \vec \fy - S(-\tau_m) \vec \phi_{m, L})
\end{align*}
is a solution to the linear wave equation, we can use the monotonicity of the energy on exterior cones to deduce that 
\begin{align*}
\|\vec \fy_L(t)- \vec \phi_{m, L}(t)\|_{H \times L^2 (r \ge t- \tau_m+ \be_m \la_m)}^2 \to 0 \quad \textrm{as} \quad t \to \infty.
\end{align*}
Combining the above with \eqref{scat fix m} we can conclude that 
\begin{align*}
\|\vec \psi(t) - (\pi, 0) - \vec \fy_{L}(t)\|_{H \times L^2 (r \ge t- \tau_m+ \be_m \la_m)}^2 \to 0 \quad \textrm{as} \quad t \to \infty.
\end{align*}
The above holds for each $m \in \N$  and for {\em any} sequence $\be_m  \to \infty$ with $\be_m <c_0\al_m$. Taking $\be_m \ll \al_m$ and recalling that $\tau_m \to \infty$ and $\la_m$ are such that $\al_m \la_m \ll \tau_m$ we have that $\tau_m- \be_m \la_m \to \I$ as $m \to \infty$.  Therefore, for any $A>0$ we can find $m$ large enough so that $\tau_m- \be_m \la_m>A$, which proves \eqref{rad} in light of the above. 

It remains to show \eqref{fy 2eq}. But this follows immediately from the decomposition \eqref{prof for fy} and the almost orthogonality of the nonlinear wave map energy for such a decomposition, see \cite[Lemma $2.16$]{CKLS1}, since we know that the left-hand-side of \eqref{prof for fy} satisfies 
\begin{align*}
\E( \vec \fy_n) \le C <2\E(Q)
\end{align*}
for large enough $n$. 
\end{proof}

Now that we have constructed the radiation term $\vec \fy_L(t)$ we denote by $\fy(t) \in \HH_0$ the global wave map that scatters to the linear wave $\vec\fy_L(t)$, i.e., $\vec \fy(t) \in \HH_0$ is the global solution to \eqref{cp} such that 
\begin{align}\label{fy fyL scat}
\| \vec\fy(t)- \vec \fy_L(t)\|_{H \times L^2} \to 0 \quad \textrm{as} \quad t  \to \I.
\end{align}
The existence of such a $\fy(t) \in \HH_0$ locally around  $t= + \infty$ follows immediately from the existence of wave operators for the corresponding $4d$ semi-linear equation. The fact that $\fy(t)$ is global-in-time follows from \cite[Theorem $1$]{CKLS1} since \eqref{fy 2eq} and \eqref{fy fyL scat} together imply that $\E(\vec \fy) <2\E(Q)$. 

We will need a few facts about the degree zero wave map $\vec \fy(t)$ which we collect in the following lemma. 

\begin{lem} Let $\vec \fy(t)$ be defined as above. Then we have 
\begin{align}
&\limsup_{t \to \I} \|\vec \fy(t)\|_{H \times L^2( \abs{r-t} \ge A)} \to 0 \quad \textrm{as} \quad A \to \infty\label{channel},\\
& \lim_{t \to \infty} \E_{t-A}^{\infty}( \vec \fy(t)) \to \E(\vec \fy) \quad \textrm{as} \quad A \to \infty \label{ext en A fy}.
\end{align}
\end{lem}
\begin{proof}First we prove \eqref{channel}. We have
\begin{align*}
\|\vec \fy(t)\|_{H \times L^2( \abs{r-t} \ge A)}^2 \le \| \vec\fy(t)- \vec \fy_L(t)\|_{H \times L^2}^2  + \|\fy_L(t)\|_{H \times L^2( \abs{r-t} \ge A)}^2.
\end{align*}
By \eqref{fy fyL scat} the first term on the right-hand-side above tends to $0$ as $t \to \infty$ so it suffices to show that 
\begin{align*}
\limsup_{t \to \I}  \|\fy_L(t)\|_{H \times L^2( \abs{r-t} \ge A)}^2 \to 0 \quad \textrm{as} \quad A \to \infty.
\end{align*}
Since $\fy_L(t)$ is a solution to \eqref{2d lin} the above follows from \cite[Theorem $4$]{CKS} by passing to the analogous statement for the corresponding $4d$ free wave $v_L(t)$  given by $$r v_L(t, r):= \fy_L(t, r).$$
To prove \eqref{ext en A fy} we note that the limit as $t \to \infty$ exists for any fixed $A$ due to the monotonicity of the energy on exterior cones. Next observe that we have 
\begin{align}\label{inner en fy}
\lim_{t \to \infty}\E_0^{t-A}( \vec \fy(t)) \le \lim_{t \to \infty}\|\vec \fy(t)\|_{H \times L^2( r \le t-A)}^2 \to 0 \quad \textrm{as} \quad A \to \infty
\end{align}
by \eqref{channel} and then \eqref{ext en A fy} follows immediately from the conservation of energy. 
\end{proof}
Now, observe that we can combine Proposition~\ref{rad prop} and \eqref{fy fyL scat} to conclude that for all $A  \ge 0$ we have
\begin{align}\label{da supp1}
\| \vec \psi(t)- (\pi, 0) - \vec \fy(t) \|_{H \times L^2( r\ge t-A)} \to 0 \quad \textrm{as} \quad t \to \I.
\end{align}
With this in mind we define $a(t)$ as follows: 
\begin{align}\label{a def}
\vec a(t) := \vec \psi(t)- \vec \fy(t)
\end{align}
and we aggregate some preliminary information about $a$ in the following lemma: 

\begin{lem}\label{a lem}
Let $\vec a(t)$ be defined as in \eqref{a def}. Then $\vec a(t) \in \HH_1$ for all $t$. Moreover,  
\begin{itemize}
\item for all $\la>0$  we have 
\begin{align}\label{da supp}
\|\vec a(t)-(\pi, 0) \|_{H \times L^2( r \ge \la t)} \to 0 \quad \textrm{as} \quad  t \to \I,
\end{align}
\item the quantity $\E(\vec a(t) )$ has a limit as $t \to \infty$ and 
\begin{align}\label{ea}
\lim_{t \to \infty}\E(\vec a(t)) = \E(\vec \psi)- \E(\vec \fy).
\end{align}
\end{itemize}
\end{lem}
\begin{proof}
By definition we have $a(t) \in \HH_1$ for all $t$ since $$a(t, 0)=0, \, \, a(t, \infty)= \pi.$$ To prove \eqref{da supp} observe that for every $A \le (1- \la)t$ we have
\begin{align*}
\|\vec a(t)-(\pi, 0) \|_{H \times L^2( r \ge \la t)}^2 &\le \|\vec \psi(t)-(\pi, 0) \|_{H \times L^2( \la t \le r\le t-A)}^2 \\
& \quad+ \|\vec\fy(t)\|_{H \times L^2(\la t \le r \le t-A)}^2\\
& \quad +\|\vec a(t)-(\pi, 0) \|_{H \times L^2( r \ge t-A)}^2.
\end{align*}
Then \eqref{da supp} follows by combining \eqref{da supp1}, \eqref{channel}, and \eqref{limsup H}. To prove \eqref{ea} we first claim that 
\begin{align}\label{epsi=efy}
\lim_{A \to \infty} \lim_{t \to \infty} \E_{t-A}^{\infty}( \vec \psi(t)) = \E(\vec \fy).
\end{align}
Indeed, we have
\begin{align*}
 \E_{t-A}^{\infty}( \vec \psi(t)) &=  \int_{t-A}^{\infty} [(\psi_t(t)- \fy_t(t) + \fy_t(t))^2+(\psi_r(t)- \fy_r(t) + \fy_r(t))^2] \, r\, dr\\
 &\quad + \int_{t-A}^{\I} \frac{\sin^2( \psi(t)- \pi - \fy(t) + \fy(t))}{r} \, dr\\
 &= \E_{t-A}^{\infty}( \vec \fy(t)) + \| \vec \psi(t)-(\pi, 0) - \vec \fy(t)\|_{\dot{H}^1 \times L^2( r \ge t-A)}^2\\
 &\quad+O\left( \| \vec \psi(t)-(\pi, 0) - \vec \fy(t)\|_{\dot{H}^1 \times L^2( r \ge t-A)}\| \vec \fy(t)\|_{\dot{H}^1 \times L^2( r \ge t-A)}\right)\\
 &\quad + \int_{t-A}^{\infty}  \frac{\sin^2( \psi(t)- \pi - \fy(t) + \fy(t))- \sin^2(\fy(t))}{r} \, dr\\
 &= \E_{t-A}^{\infty}( \vec \fy(t)) + O\left(\| \vec \psi(t)-(\pi, 0) - \vec \fy(t)\|_{H \times L^2( r \ge t-A)}^2\right)\\
 &\quad +O\left( \sqrt{\E(\vec \fy)}\| \vec \psi(t)-(\pi, 0) - \vec \fy(t)\|_{H\times L^2( r \ge t-A)}\right),
 \end{align*}
 which proves \eqref{epsi=efy} in light of \eqref{ext en A fy} and \eqref{da supp1}. In the third equality above we have used the simple trigonometric inequality: 
 \begin{align*}
 \abs{\sin^2(x-y+y)- \sin^2(y)} \le 2\abs{\sin(y)} \abs{x-y} + 2 \abs{x-y}^2.
 \end{align*}
 Now, fix $\de>0$. By \eqref{channel}, \eqref{epsi=efy}, and \eqref{da supp1} we can choose $A, T_0$ large enough so that for all $t \ge T_0$ we have
 \begin{align*}
& \| \vec \fy(t)\|_{ H \times L^2(r \le t-A)}  \le \de,\\
 &\abs{\E_{t-A}^{\I}( \vec \psi(t))- \E(\vec \fy)} \le \de,\\
 & \|\vec a(t)-(\pi, 0)\|_{H \times L^2(r \ge t-A)}^2 \le \de.
 \end{align*}
 Then for all $t \ge T_0$ and $A$ as above we can argue as before to obtain 
 \begin{align*}
 \E( \vec a(t))&= \E_0^{t-A}( \vec a(t)) + O( \|\vec a(t) -(\pi, 0)\|_{H \times L^2(r \ge t-A)}^2\\
 &=  \E_0^{t-A}( \vec \psi(t)) + O\left( \sqrt{\E(\vec \psi)}\|\vec \fy(t)\|_{H \times L^2(r \le t-A)}\right)\\
 & \quad + O\left( \|\vec \fy(t)\|_{H \times L^2(r \le t-A)}^2\right) + O\left( \|\vec a(t) -(\pi, 0)\|_{H \times L^2(r \ge t-A)}^2\right)\\
 &=  \E( \vec \psi)-   \E^{\infty}_{t-A}( \vec \psi(t)) +O(\de)\\
 &= \E( \vec \psi)-   \E( \vec \fy) +O(\de),
 \end{align*}
 which proves \eqref{ea}.
 \end{proof}

We will also need the following technical lemma in the next section.
\begin{lem} \label{tech lem} For any sequence $\s_n>0$ with $\la_n \ll \s_n \ll \tau_n$ we have 
\begin{align}
\lim_{n \to \infty} \frac{1}{\s_n} \int_{\tau_n}^{\tau_n + \s_n} \int_0^{\infty} \dot a^2( t, r) \, r \, dr \, dt =0.
\end{align}
\end{lem}
\begin{proof}
Fix  $0<\la<1$.  For each $n $ we have
\begin{align*}
\frac{1}{\s_n} \int_{\tau_n}^{\tau_n + \s_n} \int_0^{\infty} \dot a^2( t, r) \, r \, dr \, dt & \le \frac{1}{\s_n} \int_{\tau_n}^{\tau_n + \s_n} \int_0^{\la t} \dot a^2( t, r) \, r \, dr \, dt \\
& \quad + \frac{1}{\s_n} \int_{\tau_n}^{\tau_n 
 + \s_n} \int_{\la t}^{\infty} \dot a^2( t, r) \, r \, dr \, dt .
\end{align*}
By \eqref{da supp} we can conclude that 
\begin{align*}
\lim_{n \to \infty} \sup_{t \ge \tau_n} \int_{\la t}^{\infty} \dot{a}^2(t, r) \, r \, dr  =0.
\end{align*}
Hence it suffices to show that 
\begin{align*}
\lim_{n \to \infty} \frac{1}{\s_n} \int_{\tau_n}^{\tau_n + \s_n} \int_0^{\la t} \dot a^2( t, r) \, r \, dr \, dt =0.
\end{align*}
Observe that for every $n$ we have 
\begin{align} \label{a le psi fy}
\frac{1}{\s_n} \int_{\tau_n}^{\tau_n + \s_n} \int_0^{\la t} \dot a^2( t, r) \, r \, dr \, dt  &\lesssim \frac{1}{\s_n} \int_{\tau_n}^{\tau_n + \s_n} \int_0^{\la t} \dot \psi^2( t, r) \, r \, dr \, dt \\
&\quad+\frac{1}{\s_n} \int_{\tau_n}^{\tau_n + \s_n} \int_0^{\la t} \dot \fy^2( t, r) \, r \, dr \, dt.  \notag
\end{align}
We first estimate the first integral on the right-hand-side above. Let $A_n \to \infty$ be the sequence in Proposition~\ref{str diag}, see also Remark~\ref{rem str}, and let $t_n \to \infty$ be the sequence in Theorem~\ref{struwe}. Recall that we have $\tau_n \in [t_n, t_n + \la_n]$ and $\la_n \le A_n \ll t_n$. 

Observe that for $n$ large enough we have that for each $t \in [ \tau_n, \tau_n + \s_n]$ we have $\la t \le t- A_n$.  Hence, 
\begin{align*}
\frac{1}{\s_n} \int_{\tau_n}^{\tau_n + \s_n} \int_0^{\la t} \dot \psi^2( t, r) \, r \, dr \, dt  \le \frac{1}{\s_n} \int_{\tau_n}^{\tau_n + \s_n} \int_0^{t- A_n} \dot \psi^2( t, r) \, r \, dr \, dt.
\end{align*}
Next, note that since $\la_n \ll \s_n $ we can ensure that for $n$ large enough we have $\la_n + \s_n \le 2 \s_n$. Therefore, 
\begin{align*}
\frac{1}{\s_n} \int_{\tau_n}^{\tau_n + \s_n} \int_0^{t- A_n} \dot \psi^2( t, r) \, r \, dr \, dt \le  \frac{2}{\la_n+ \s_n} \int_{t_n}^{t_n + \la_n+ \s_n} \int_0^{t- A_n} \dot \psi^2( t, r) \, r \, dr \, dt \to 0
\end{align*}
as $n \to \infty$ by Lemma~\ref{sig}.% since $\la_n + \s_n \ll t_n$. 

Lastly we estimate the second integral on the righ-hand-side of \eqref{a le psi fy}. For each $A>0$ we can choose $n$ large enough so that $\la t \le t-A$ for each $t \in [\tau_n, \tau_n + \s_n]$. So we have
 \begin{align*}
 \frac{1}{\s_n} \int^{\tau_n+ \s_n}_{\tau_n} \int_0^{\la t} \dot \fy^2(t, r) \, r\, dr\, dt \le  \frac{1}{\s_n} \int^{\tau_n+ \s_n}_{\tau_n} \int_0^{t-A} \dot \fy^2(t, r) \, r\, dr\, dt.
 \end{align*}
 Taking the limsup as $n  \to \infty$ of both sides and then letting $A \to \infty$ on the right we have by \eqref{channel} that the left-hand-side above tends to $0$ as $n \to \infty$. This concludes the proof. 
\end{proof}

%----------------------------------------Compactness of error---------------------------------------------------------------%

\subsection{Compactness of the error}
For the remainder of this section, we fix $\al_n \to \infty$ and find $\tau_n \to \infty$ and $\la_n \ll \tau_n$ as in Proposition~\ref{str diag}. We define $\vec b_n =(b_{n, 0}, b_{n, 1})\in \HH_0$ as follows:
\begin{align}\label{bn def}
&b_{n, 0}(r):= a(\tau_n, r)- Q(r/ \la_n),\\
&b_{n, 1}(r):= \dot{a}(\tau_n, r)\label{bn1 def}.
\end{align}
As in \cite[Section $5.3$]{CKLS1}, our goal in this subsection is to show that $\vec b_n \to 0$ in the energy space. Indeed we prove the following result: 

\begin{prop}\label{b compact}
Define $\vec b_n \in \HH_0$ as in \eqref{bn def}, \eqref{bn1 def}. Then, 
\begin{align}
\|\vec b_n \|_{H \times L^2} \to 0 \quad \textrm{as} \quad n \to \infty.
\end{align}
%as $n \to \infty$. 
\end{prop}

\begin{rem} In light of \eqref{fy fyL scat}, it is clear  that Proposition~\ref{b compact} implies Proposition~\ref{global}. 
\end{rem}

\begin{rem} The proof of Proposition~\ref{b compact} will follow the same  strategy as~\cite[Proposition~$5.6$]{CKLS1} and we refer the reader to the outline given there for a general overview of the proof. 
\end{rem}

We begin with the following consequences of the previous sections.  

\begin{lem} \label{b prelim} Let $\vec b_n \in \HH_0$ be defined as above. Then we have 
\begin{itemize}
\item[(a)] As $n \to \I$ we have 
\begin{align}\label{bn1 to 0}
 \|b_{n, 1}\|_{L^2} \to 0.
 \end{align}
 \item[(b)] As $n \to \infty$ we have 
 \begin{align}\label{bn0 int}
 \|b_{n, 0}\|_{H( r \le \al_n \la_n)} \to 0.
 \end{align}
 \item[(c)] For any fixed $\la>0$ we have 
 \begin{align}\label{bn0 ext}
 \|b_{n, 0}\|_{H ( r \ge \la \tau_n)} \to 0 \quad \textrm{as} \quad n \to \infty.
 \end{align}
 \item[(d)] There exists a $C>0$ so that 
 \begin{align}\label{e bn}
 \E(\vec b_n) \le C < 2\E(Q)
 \end{align}
 for $n$ large enough. 
 \end{itemize}
 \end{lem}

\begin{proof} To prove \eqref{bn1 to 0} fix $0 < \la< 1$ and observe that %for fixed $A>0$ 
we have 
\begin{align*}
\int_{0}^{\infty} b_{n, 1}^2(r) \, r\, dr &\le  \int_0^{\la \tau_n} \dot{\psi}^2(\tau_n, r) \, r\, dr
 + \int_0^{ \la \tau_n} \dot{\fy}^2(\tau_n, r) \, r\, dr \\
 &\quad+ \int_{ \la \tau_n}^{\infty} \dot{a}(\tau_n, r)^2 \, r \, dr.
 \end{align*}
Then \eqref{bn1 to 0} follows from \eqref{psi dot}, \eqref{channel}, and \eqref{da supp}. 

Next we prove \eqref{bn0 int}. To see this, observe that for each $n$ we have
\begin{align*}
\|b_{n, 0}\|_{H( r \le \al_n \la_n)}^2 \le \| \psi(\tau_n)- Q(\cdot/ \la_n)\|_{H (r \le \al_n \la_n)}^2 + \|\fy(\tau_n)\|_{H(r \le \al_n \la_n)}^2.
\end{align*}
The first term on the right-hand-side tends to zero as $n \to \infty$ by \eqref{psi-Q}. To estimate the second term on the right-hand-side we note that for fixed $A >0$ we can find $n$ large enough so that $\al_n \la_n \le \tau_n -A$ and so we have 
\begin{align*}
\|\fy(\tau_n)\|_{H(r \le \al_n \la_n)}^2 \le \|\fy(\tau_n)\|_{H(r \le \tau_n- A)}^2.
\end{align*}
Taking the limsup as $n \to \infty$ on both sides above and then taking $A \to \infty$ on the right and recalling \eqref{channel}  we see that the limit as $n \to \infty$ of the left-hand side above must be zero. This proves \eqref{bn0 int}. 

To deduce \eqref{bn0 ext} note that 
\begin{align*}
 \|b_{n, 0}\|_{H ( r \ge \la \tau_n)}^2 \le \|a(\tau_n)- \pi\|_{H (r \ge \la\tau_n)}^2 +\|Q(\cdot/ \la_n)- \pi\|_{H( r \ge \la \tau_n)}^2.
 \end{align*}
 The first term on the right-hand-side above tends to zero as $ n \to \I$ by \eqref{da supp}. The second term tends to zero since $ \la \tau_n/ \la_n \to \infty$ as $n \to \I$. 
 
 Finally, we establish \eqref{e bn}. First observe that for any fixed $\la>0$, \eqref{bn0 ext} implies that 
 \begin{align*}
 \E(\vec b_n)&= \E_0^{\la \tau_n}(\vec b_n) + \E_{\la \tau_n}^{\infty} (\vec b_n)\\
 &=  \E_0^{\la \tau_n}(\vec b_n) +o_n(1)
 \end{align*}
 as $n \to \infty$. So it suffices to control $\E_0^{\la \tau_n}(\vec b_n)$. Next, observe that for $n$ large enough, \eqref{inner en fy} gives that 
 \begin{align*}
\|\vec \fy(\tau_n)\|_{H  \times L^2(r \le \la \tau_n)}\le \|\vec \fy(\tau_n)\|_{H  \times L^2(r \le \tau_n- A)} 
\end{align*}
and the right-hand side is small for $n, A$ large. This means that the contribution of $\vec \fy(\tau_n)$ is negligible on $r \le \la \tau_n$, and thus 
\begin{align*}
\E_0^{\la \tau_n}(\vec b_n) = \E_0^{\la \tau_n}( \vec \psi(\tau_n) - (Q(\cdot/ \la_n), 0)) + o_n(1).
\end{align*}
Next, recall that Proposition~\ref{str diag} implies that 
\begin{align}\label{str diag en}
\E_0^{\al_n \la_n}( \vec \psi(\tau_n)-Q (\cdot/ \la_n), 0) = o_n(1),
\end{align}
which shows in particular that 
\begin{align}\label{en psi ext}
\E_{\al_n \la_n}^{\infty}(\vec \psi(\tau_n)) \le  \eta +o_n(1)
\end{align}
where $\eta:= \E(\vec \psi)- \E(Q)<2 \E(Q)$. Also, \eqref{str diag en}  means that it suffices to show that 
\begin{align*}
\E_{\al_n\la_n}^{\la \tau_n}(\vec \psi(\tau_n)-(Q (\cdot/ \la_n), 0) ) \le C <2\E(Q).
\end{align*}
Note that since $ \al_n \to \infty$ we have 
\begin{align*}
\E_{\al_n \la_n}^{\infty}(Q(\cdot/ \la_n))= \E_{\al_n}^{\I}(Q) =o_n(1).
\end{align*}
Hence,
\begin{align*}
\E_{\al_n\la_n}^{\la \tau_n}(\vec \psi(\tau_n)-(Q (\cdot/ \la_n), 0) )  = \E_{\al_n\la_n}^{\la \tau_n}(\vec \psi(\tau_n) ) +o_n(1) \le \eta +o_n(1),
\end{align*}
which completes the proof. 
\end{proof}
Next, we would like to show that the sequence $\vec b_n$ does not contain any nonzero profiles. This next result is the global-in-time analog of \cite[Proposition $5.7$]{CKLS1} and the proof is again, reminiscent of the the arguments given in \cite[Section $5$]{DKM1}.  

Denote by $\vec b_n(t) \in \HH_0$ the wave map evolution with data $\vec b_n$. By \eqref{e bn} and \cite[Theorem $1.1$]{CKLS1} we know that $\vec b_n(t) \in \HH_0$ is global in time and scatters to zero as $t \to \pm \infty$. 

The statements of the following proposition and its corollary are identical to the corresponding statements \cite[Proposition $5.7$ and Corollary $5.8$]{CKLS1} in the finite time blow-up case.

\begin{prop}\label{no prof} Let $b_n \in \HH_0$ and the corresponding global wave map evolution $\vec b_n(t) \in \HH_0$ be defined as above. Then, there exists a decomposition
\begin{align}\label{bn decamp}
\vec b_n(t, r)= b_{n, L}(t, r) + \vec \theta_n(t, r)
\end{align}
where $\vec b_{n, L}$ satisfies the linear wave equation \eqref{2d lin} with initial data $\vec b_{n, L}(0, r):=(b_{n, 0}, 0)$. Moreover, $b_{n, L}$ and $\vec \theta_n$ satisfy
\begin{align}
&\left\|\frac{1}{r} b_{n, L}\right\|_{L^3_t(\R; L^6_x(\R^4))} \longrightarrow 0\label{bL to 0}\\
&\|\vec \theta_n\|_{L^{\infty}_t(\R; H\times L^2)} + \left\| \frac{1}{r} \theta_n\right\|_{L^3_t(\R ; L^6_x(\R^4))} \longrightarrow 0 \label{theta to 0}
\end{align}
as $n \to \infty$. 
\end{prop}

Before beginning the proof of Proposition~\ref{no prof} we use the conclusions of the proposition to deduce the following  corollary  which will be an essential ingredient in the proof of Proposition~\ref{b compact}.
\begin{cor} \label{ext en est}
Let $\vec b_n(t)$ be defined as in Proposition~\ref{no prof}. Suppose that there exists a constant $\de_0$ and a subsequence in $n$ so that $\|b_{n, 0}\|_H \ge \de_0$.  Then there exists $\al_0 >0$ such that for all $t>0$ and all $n$ large enough along this subsequence we have
\begin{align} \label{b ext}
\|\vec b_n(t) \|_{H \times L^2(r \ge t)} \ge \al_0 \de_0.
\end{align}
\end{cor}
\begin{proof} First note that since $\vec b_{n, L} $ satisfies the  linear wave equation \eqref{2d lin} with initial data $\vec b_{n, L}(0) =(b_{n, 0}, 0)$ we know by  \cite[Corollary $5$]{CKS} and \cite[Corollary $2.3$]{CKLS1},  that there exists a constant $\be_0 >0$ so that for each $t\ge 0$ we have
\begin{align*} 
\|\vec b_{n, L}(t)\|_{H \times L^2(r \ge t)} \ge \be_0\|b_{n, 0}\|_H.
\end{align*}
On the other hand, by Proposition~\ref{no prof} we know that 
\begin{align*} 
\|\vec b_{n}(t) - \vec b_{n, L}(t)\|_{H \times L^2(r\ge t)} \le \|\vec \theta_n(t)\|_{H \times L^2} = o_n(1).
\end{align*}
Putting these two facts together gives 
\begin{align*} 
\|\vec b_n(t) \|_{H \times L^2(r \ge t)} &\ge \|b_{n, L}(t)\|_{H \times L^2(r \ge t)}  -o_n(1)\\
&\ge  \be_0\|b_{n, 0}\|_H -o_n(1).
\end{align*}
This yields \eqref{b ext} by passing to a suitable subsequence and taking $n$ large enough.
\end{proof}

The proof of  Proposition~\ref{no prof} is very similar to the proof of  \cite[Proposition $5.7$]{CKLS1}. Instead of going through the entire argument again here, we establish the main ingredients of the proof and we refer the reader to \cite{CKLS1} for the remainder of the argument. 

Since $\vec b_n \in \HH_0$ and  $\E(\vec  b_n) \le C < 2\E(Q)$ we can, by \cite[Corollary $2.15$]{CKLS1}, consider the following profile decomposition for $\vec b_n$: 
\begin{align}\label{bg1}
&b_{n, 0}(r)= \sum_{j \le k}\fy_L^j\left(\frac{-t_n^j}{\la_n^j}, \frac{r}{\la_n^j}\right) + \ga_{n, 0}^k( r),\\
&b_{n, 1}(r) = \sum_{j \le k} \frac{1}{\la_n^j} \dot{\fy}_L^j\left(\frac{-t_n^j}{\la_n^j}, \frac{r}{\la_n^j}\right) + \ga_{n, 1}^k( r)\label{bg2},
\end{align}
where each $\vec \fy^j_L$ is a solution to \eqref{2d lin} and where we have for each $j \neq k$:
\begin{align} \label{ort}
\frac{\la_n^j}{\la_n^k} + \frac{\la_n^k}{\la_n^j} + \frac{\abs{t_n^j -t_n^k}}{\la_n^k} +  \frac{\abs{t_n^j -t_n^k}}{\la_n^j} \to \infty \quad \textrm{as} \quad n \to \infty.
\end{align} 
Moreover, if we denote by $\vec \ga_{n, L}^k(t)$ the linear evolution of $\vec \ga_n^k$, i.e., solution to \eqref{2d lin}, we have for $j\le k$ that
\begin{align} 
&\left( \ga_{n,L}^k(\la_n^jt_n^j, \la_n^j \cdot), \la_n^j \dot \ga_{n,L}^k(\la_n^j t_n^j, \la_n^j \cdot)\right) \rightharpoonup 0 \, \, \textrm{in} \, \,  H \times L^2 \quad \textrm{as} \quad n \to \infty \label{weak}\\
&\limsup_{n \to \infty} \left\|\frac{1}{r} \ga_{n,L}^k\right\|_{L^3_tL^6_x(\R^4)} \to 0 \quad\textrm{as} \quad k \to \infty.\label{stric}
\end{align}
Finally we have the following Pythagorean expansions: 
\begin{align} \label{pyt}
\| \vec b_{n}\|_{H \times L^2}^2 = \sum_{j \le k}\left\|  \vec \fy_L^j\left(\frac{-t_n^j}{\la_n^j}\right)\right\|_{H \times L^2}^2 + \|\vec \ga_{n}^k\|_{H \times L^2}^2 + o_n(1)%\\
%&\|b_{n, 1}\|_{L^2}^2 = \sum_{j \le k}\left\| \dot \fy_L^j\left(\frac{-t_n^j}{\la_n^j}\right)\right\|_{L^2}^2 + \|\ga_{n, 1}^k\|_{L^2}^2 + o_n(1) \label{pyt2}.
\end{align}
As in \cite{CKLS1}, the proof of Proposition~\ref{no prof} will consist of a sequence of steps designed to show that each of the profiles $\fy_L^j$ must be identically zero. Arguing exactly as in \cite[Lemma $5.9$, Corollary $5.10$]{CKLS1} we can first deduce that the times $t_n^j$ can be taken to be $0$ for each $n, j$ and that the the initial velocities $\dot \fy^j_L(0)$ must all be identically zero as well.  We summarize this conclusion in the following lemma: 

\begin{lem} \label{in vel 0}
In the decomposition \eqref{bg1}, \eqref{bg2} we can assume, without loss of generality, that $t_n^j =0$ for every $n$ and for every $j$. In addition, we then have 
\begin{align*} 
\dot{\fy}_L^j(0, r) \equiv 0 \quad \textrm{for every} \quad  j.
\end{align*} 
\end{lem}
The proof of Lemma~\ref{in vel 0} is identical to the proof of \cite[Lemma $5.9$]{CKLS1} and the proof of~\cite[Corollary 5.10]{CKLS1}. %and follows from the Pythagorean expansion \eqref{pyt2}, \eqref{bn1 to 0}, and the asymptotic equipartition of energy for the corresponding $4d$ free waves. 
We refer the reader to \cite{CKLS1} for the details. 

By Lemma~\ref{in vel 0} we can rewrite our profile decomposition as follows: 
\begin{align}\label{bg bn}
&b_{n, 0}(r)= \sum_{j \le k}\fy_L^j\left(0, r/\la_n^j\right) + \ga_{n, 0}^k( r)\\
&b_{n, 1}(r) =o_n(1) \, \, \textrm{in} \, \, L^2 \, \,  \textrm{as} \, \, n \to \infty\label{bg bn1},
\end{align}
Note that in addition to the Pythagorean expansion in \eqref{pyt} we also have the following almost-orthogonality of the nonlinear wave map energy, which was established in \cite[Lemma $2.16$]{CKLS1}: 
\begin{align} \label{nonlin en}
\E(\vec b_n) = \sum_{j\le k} \E(\fy^j_L(0), 0) + \E(\ga_{n, 0}^k, 0) + o_n(1).
\end{align}
Note that $\fy^j, \ga_{n, 0}^k \in \HH_0$ for every $j$, for every $n$, and for every $k$.  Using the fact that $\E(\vec b_n) \le C< 2\E(Q)$, \eqref{nonlin en} and \cite[Theorem $1.1$]{CKLS1} imply that,  for every $j$, the nonlinear wave map evolution of the data $(\fy^j_L(0, r/ \la_n^j), 0)$  given by
\begin{align} 
\vec \fy_n^j(t, r) := \left (\fy^j\left(\frac{t}{\la_n^j}, \frac{r}{\la_n^j}\right)\, , \, \frac{1}{\la_n^j} \dot{\fy}^j\left(\frac{t}{\la_n^j}, \frac{r}{\la_n^j}\right)\right)
\end{align}
is global in time and scatters as $t\to \pm \infty$. Moreover we have the following nonlinear profile decomposition guarranteed by  \cite[Proposition $2.17$]{CKLS1}:
\begin{align} \label{nonlin prof}
\vec b_n(t, r)  = \sum_{j \le k} \vec \fy_n^j(t, r) + \vec \ga_{n, L}^k(t, r) + \vec \theta_n^k(t, r)
\end{align} 
where the $\vec b_n(t, r)$ are the global wave map evolutions of the data $\vec b_n$, $\vec \ga_{n,  L}^k(t, r)$ is the linear evolution of $(\ga_n^k, 0)$, and the errors $\vec \theta_n^k$ satisfy
\begin{align} 
& \limsup_{n\to \infty}\left(\|\vec \theta_n^k\|_{L^{\infty}_t(H\times L^2)} + \left\| \frac{1}{r} \theta_n^k\right\|_{L^3_t(\R ; L^6_x(\R^4))}\right) \to 0 \quad \textrm{as} \quad k \to \infty\label{theta 0}.
\end{align}

Recall that we are trying to show that all of the profiles $\fy^j$ must be identically equal to zero. As in \cite{CKLS1} we can make the following crucial observations about the scales $\la_n^j$. Since we have concluded that we can assume that  all of the times $t_n^j=0$ for all $n, j$ we first note that the orthogonality condition \eqref{ort} implies that for $j \neq k$:
\begin{align*}
\frac{\la_n^j}{ \la_n^k} + \frac{\la_n^k}{\la_n^j} \to \infty \, \, \textrm{as} \, \, n \to \infty.
\end{align*}
Next, recall that by Lemma~\ref{b prelim} we have 
\begin{align}\label{b int}
&\|b_{n, 0} \|_{H( r \le \al_n \la_n)} \to 0 \, \, \textrm{as} \, \, n \to \infty,\\
&\|b_{n, 0} \|_{H( r \ge \la \tau_n)} \to 0 \, \, \textrm{as} \, \, n \to \infty, \, \, \forall \la>0 \, \, \textrm{fixed} \label{b ext1}.
\end{align}
Combining the above two facts with \cite[Proposition $2.19$]{CKLS1} we can conclude that for each $\la_n^j$ corresponding to a nonzero profile $\fy^j$ we have 
\begin{align}\label{scales}
\la_n \ll \la_n^j \ll \tau_n \quad \textrm{as} \quad n \to \infty.
\end{align}

Now, let $k_0$ be the index corresponding to the first nonzero profile $\fy^{k_0}$. We can assume, without loss of generality that $k_0=1$. By \eqref{b int}, \eqref{scales} and \cite[Appendix $B$]{DKM1} we can find a sequence $\ti \la_n$ so that
\begin{align*}
&\ti \la_n \ll \al_n \la_n\\
&\la_n  \ll  \ti \la_n \ll \la_n^1\\
&\ti \la_n \ll \la_n^j \, \, \textrm{or} \,\, \la_n^j \ll \ti \la_n \quad \forall j>1.
\end{align*}
Define
\begin{align*}
\be_n= \frac{\ti \la_n}{\la_n} \to \infty
\end{align*}
and we note that $\be_n  \ll \al_n$ and $\ti\la_n = \be_n  \la_n$. Therefore, up to replacing $\be_n$ by a sequence $\ti \be_n \simeq \be_n$  and $\ti \la_n$ by $\ti{\ti \la}_n:= \ti \be_n \la_n$, we have by Corollary~\ref{tech} and a slight abuse of notation that
\begin{align}\label{psi ti la to pi}
\psi(\tau_n, \ti \la_n) \to \pi \quad \textrm{as} \quad n \to \infty.
\end{align}
We define the set 
\begin{align*} 
\J_{\textrm{ext}}:= \{ j \ge 1 \, \, \vert \, \,  \ti\la_n \ll \la_n^j\}.
\end{align*}
Note that by construction $1 \in \J_{\textrm{ext}}$. 

The above technical ingredients are necessary for the proof of the following lemma and its corollary. The analog in the finite-time blow-up case is  \cite[Lemma $5.10$]{CKLS1}. 
\begin{lem} \label{2 lim} Let $\fy^1$, $\la_n^1$ be defined as above. Then for all $\e>0$ we have 
\begin{align}\label{p1}
&\frac{1}{\la_n^1} \int_0^{\la_n^1} \int_{ \e \la_n^1 + t}^{\infty}  \abs{  \sum_{ j \in \J_{\textrm{ext}} \, , j \le k }  \dot\fy_n^j(t, r) + \dot \ga_{n, L}^k(t, r)}^2 \, r\, dr\, dt= o^k_n 
\end{align}
where $\ds{\lim_{k\to \infty} \limsup_{n \to \infty} o^k_n =0}$. 
Also, for all $j>1$ and for all $\e>0$ we have 
\begin{align} \label{p2}
\frac{1}{\la_n^1} \int_0^{\la_n^1} \int_{ \e \la_n^1 + t}^{\infty} ( \dot \fy_n^j)^2(t, r) \, r\, dr dt \to 0 \quad \textrm{as} \quad n \to \infty.
\end{align}
\end{lem}
\begin{rem} We refer the reader to \cite[Proof of Lemma $5.10$]{CKLS1} for the details of the proof of Lemma~\ref{2 lim}. The proof of \eqref{p1} is nearly identical to \cite[Proof of $(5.57$)]{CKLS1} the one difference being that here we use Lemma~\ref{tech lem} in place of the argument following \cite[equation ($5.66$)]{CKLS1}. The proof of \eqref{p2} is identical to  \cite[Proof of ($5.58$)]{CKLS1} and we omit it here. 
\end{rem}

Note that \eqref{p1} and \eqref{p2} together directly imply the following result: 
\begin{cor}\label{p3}Let $\fy^1$ be as in Lemma~\ref{2 lim}. Then for all $\e>0$ we have 
\begin{align} 
\frac{1}{\la_n^1} \int_0^{\la_n^1} \int_{\e \la_n^1 + t}^{\infty}  \abs{   \dot\fy_n^1(t, r) + \dot \ga_{n, L}^k(t, r)}^2 \, r\, dr\, dt =o_n^k
\end{align}
where  $\ds{\lim_{k\to \infty} \limsup_{n \to \infty} o^k_n =0}$. 
\end{cor}

The proof of Proposition~\ref{no prof} now follows from the exact same argument  as \cite[Proof of Proposition~$5.7$]{CKLS1}. We refer the reader to \cite{CKLS1} for the details.

We can now complete the proof of Proposition~\ref{b compact}. 

\begin{proof}[Proof of Proposition~\ref{b compact}] We argue by contradiction. Assume that Proposition~\ref{b compact} fails. Then, up to extracting a subsequence, we can find a $\de_0>0$ so that 
\begin{align}\label{fail}
\|b_{n, 0}\|_{H} \ge \de_0
\end{align}
for every $n$. By Corollary~\ref{ext en est} we know that there exists $\al_0>0$ so that for all $t$,
\begin{align*}
 \|\vec b_n(t)\|_{H \times L^2(r \ge \abs{t})} \ge \al_0 \de_0.
\end{align*}
We will show that the above is, in fact, impossible by constructing a sequence of times along which the left hand side above tends to zero.
It is convenient to carry out the argument in rescaled coordinates. Set 
\begin{align*}
\mu_n:= \frac{\la_n}{\tau_n}.
\end{align*}
Since $\la_n \ll \tau_n$ as $n \to \infty$, our new scale $\mu_n \to 0 $ as $n \to \infty$.  We next define rescaled wave maps: 
\begin{align}\label{g def}
& g_n(t, r):= \psi(\tau_n + \tau_n t, \tau_n r),\\
& h_n(t, r):=\fy(\tau_n + \tau_n t, \tau_n r) \label{h def}.
\end{align}
Since $\vec g_n(t)$ and $\vec h_n(t)$ are defined by rescaling $\vec\psi$ and $\vec \fy$ we have that $\vec g_n(t) \in \HH_1$ is a global-in-time wave map and the wave map $\vec \fy(t) \in \HH_0$ is global-in-time and scatters to $0$  as $t \to \pm \infty$. 
We then have 
\begin{align*}
a(\tau_n + \tau_n t, \tau_n r) = g_n(t, r)- h_n(t, r).
\end{align*}
Similarly, we define 
\begin{align*}
&\ti{b}_{n, 0}(r):= b_{n, 0}( \tau_n r),\\
& \ti{b}_{n, 1}(r):=\tau_n b_{n, 1}( \tau_n r)
\end{align*}
 and the corresponding rescaled wave map evolutions 
 \begin{align*}
& \ti{b}_n(t, r):=b_n( \tau_n t, \tau_n r),\\
 & \p_t \ti{b}_n(t, r):= \tau_n \dot{b}_n(\tau_n t, \tau_n r).
 \end{align*}
 After this rescaling, our decomposition becomes 
\begin{align}\label{ghqb}
&g_n(0, r)= h_n(0, r) + Q\left(\frac{r}{\mu_n}\right) +\ti b_{n, 0}(r)\\
&\dot{g}_n(0, r)= \dot{h}_n(0, r) + \ti b_{n,1}(r).
\end{align}
We can rephrase \eqref{bn0 ext} and \eqref{bn0 int} in terms of this rescaling and we obtain: 
\begin{align}\label{ti b ext}
 \forall \la>0 \, \, \textrm{fixed}, \, \, &\|\ti{b}_{n, 0}\|_{H (r \ge \la)} \to 0\, \, \textrm{as} \, \, n \to \infty,\\
&\|\ti{b}_{n, 0}\|_{H ( r \le \al_n \mu_n)} \to 0\, \, \textrm{as} \, \, n \to \infty  \label{ti b int}.
\end{align}
Also, \eqref{channel} implies that 
\begin{align}
&\lim_{A \to \infty}\limsup_{n \to \infty}\|\vec h_n(0)\|_{H \times L^2( r \le 1- A/ \tau_n)} = 0\label{chan int},\\
&\lim_{A \to \infty}\limsup_{n \to \infty}\|\vec h_n(0)\|_{H \times L^2( r \ge 1+ A/ \tau_n)} = 0\label{chan ext}.
\end{align}
Next, we claim that for every $n$ a decomposition of the form \eqref{ghqb} is preserved up to a small error if we replace the terms in \eqref{ghqb} with their respective wave map evolutions at some future times to be defined precisely below.

By Corollary~\ref{tech} we can choose a sequence $\ga_n \to \infty$ with 
\begin{align*} 
\ga_n\ll \al_n
\end{align*}
so that 
\begin{align*}
g_n(0, \ga_n \mu_n) \to \pi\quad \textrm{as} \quad n \to \infty.
\end{align*}
Define $\de_n \to 0$ by 
\begin{align*} 
\abs{ g_n (0, \ga_n \mu_n) - \pi } =: \de_n \to 0.
\end{align*}
Using \eqref{psi-Q} we define $\e_n \to 0$ by 
\begin{align*}
\|\vec g_n(0)-(Q(\cdot/ \mu_n), 0)\|_{H \times L^2 ( r \le \al_n \mu_n)}=: \e_n \to 0.
\end{align*}
Finally, choose $\be_n \to \infty$ so that 
\begin{align}\notag
&\be_n \le \min\{ \sqrt{\ga_n},  \de_n^{-1/2}, \e_n^{-1/2}\}\\
&g_n(0, \be_n \mu_n/2) \to \pi \quad \textrm{as} \quad n \to \infty \label{gn bemu}.
\end{align}
As in \cite{CKLS1}, we make the following claims: 
\begin{itemize} 
\item[($i$)] As $n \to \infty$ we have 
\begin{align}\label{g-Q}
\|\vec g_n(\be_n \mu_n/2) - (Q(\cdot/ \mu_n), 0)\|_{H \times L^2( r\le \be_n \mu_n)} \to 0.
\end{align}
\item[($ii$)] For each $n$, on the interval $r \in [\be_n \mu_n, \infty)$ we have 
\begin{align}\label{ghbt}
&\vec g_n\left( \frac{\be_n \mu_n}{2}, r\right) - (\pi, 0) = \vec h_n\left( \frac{\be_n \mu_n}{2}, r\right) +  \vec{\ti{b}}_n\left( \frac{\be_n \mu_n}{2}, r\right) \\ \notag
&\quad + \vec{\breve{\theta}}_n\left( \frac{\be_n \mu_n}{2}, r\right),\\
&\|\vec{\breve{\theta}}_n\|_{L^{\infty}_t( H \times L^2)} \to 0 \notag.
\end{align}
\end{itemize}
We first prove \eqref{g-Q}. The proof is very similar to the corresponding argument in the finite-time blow-up case, see \cite[Proof of $(5.76)$]{CKLS1}. We repeat the argument here for completeness.

 %Using \eqref{chan int} together with \eqref{ti b int} and the decomposition \eqref{ghqb} 
First note that we have 
 \begin{align*}
 \|\vec g_n(0)- (Q(\cdot/ \mu_n), 0) \|_{ H \times L^2(r \le \ga_n \mu_n)}  \le  \e_n \to 0.
 \end{align*}
  Unscale the above by setting $\ti g_n(t, r) =g_n(\mu_n t, \mu_nr)$, which gives 
 \begin{align*}
  \|( \ti{g}_n(0), \p_t \ti g_n(0))- (Q(\cdot), 0) \|_{ H \times L^2(r \le \ga_n)} \le  \e_n \to 0.
  \end{align*}
  Now using \cite[Corollary $2.6$]{CKLS1} and the finite speed of propagation we claim that we have
  \begin{align} \label{ti g-Q at be}
\|( \ti{g}_n(\be_n /2), \p_t \ti g_n(\be_n /2))- (Q(\cdot), 0) \|_{ H \times L^2(r \le \be_n)} =o_n(1) .
\end{align}
 To see this, we need to show that \cite[Corollary $2.6$]{CKLS1} applies. Indeed define 
 \begin{align*}
 &\hat{g}_{n,0}(r) := \begin{cases}   \pi \quad \textrm{if} \quad r \ge 2\ga_n  \\ \pi+ \frac{\pi - \ti{g}_n(0, \ga_n)}{\ga_n}(r- 2\ga_n) \quad \textrm{if} \quad \ga_n \le r \le 2\ga_n \\ \ti g_n(0, r) \quad \textrm{if} \quad r \le \ga_n.
 \end{cases}\\
 &\hat{g}_{n, 1}(r) = \begin{cases} \p_t \ti g_n(0, r)\quad  \textrm{if} \quad r \le \ga_n \\ 0 \quad \textrm{if} \quad r \ge \ga_n\end{cases}
 \end{align*}
 Then, by construction we have $\vec{\hat{g}}_n \in \HH_1$, and since  
 \begin{align*}
 \| \vec{\hat{g}}_n-(\pi, 0)\|_{H \times L^2( \ga_n \le r \le 2\ga_n )} \le C \de_n
 \end{align*}
 we then can conclude that 
 \begin{align*} 
 \|\vec{\hat{g}}_n - (Q, 0)\|_{H \times L^2} &\le  \|\vec{\hat{g}}_n - (Q, 0)\|_{H \times L^2(r \le \ga_n)}+  \|\vec{\hat{g}}_n - ( \pi, 0)\|_{H \times L^2(\ga_n \le r \le 2\ga_n )} \\
 &\quad+  \|(\pi, 0) - (Q, 0)\|_{H \times L^2(r \ge \ga_n)}\\
 & \le C( \e_n + \de_n + \ga_n^{-1}).
 \end{align*}
 Now, given our choice of $\be_n$,  \eqref{ti g-Q at be} follows from  \cite[Corollary $2.6$]{CKLS1} and the finite speed of propagation. 
 Rescaling  \eqref{ti g-Q at be}  we have 
 \begin{align*}
\|( {g}_n(\be_n\mu_n/2), \p_t g_n(\be_n\mu_n/2))- (Q(\cdot/ \mu_n), 0) \|_{ H \times L^2(r \le \be_n\mu_n)} \to 0.
\end{align*}
This proves \eqref{g-Q}. Also note that by monotonicity of the energy on interior cones and the comparability of the energy and the $H \times L^2$  norm in $\HH_0$, for small energies, we see that  \eqref{bn1 to 0} and \eqref{ti b int} imply that 
\begin{align}\label{ti b small before be}
\|(\ti b_{n}(\be_n \mu_n/2),  \p_t\ti b_{n}(\be_n\mu_n/2))\|_{ H \times L^2(r \le \be_n \mu_n)} \to 0.
 \end{align}

Next we prove \eqref{ghbt}. First we define
\begin{align*}
&\ti g_{n, 0}(r)=  \begin{cases} \pi - \frac{\pi- g_n(0, \mu_n \be_n/2)}{\frac{1}{2}\mu_n \be_n} r\quad \textrm{if} \quad r \le  \be_n\mu_n/2\\ g_n(0, r) \quad \textrm{if} \quad  r \ge  \be_n\mu_n/2  \end{cases}\\
& \ti g_{n, 1}(r)= \dot{g}_n(0, r).
\end{align*}
Then, let $\chi \in C^{\infty}([0, \infty))$ be defined so that $\chi(r) \equiv 1$ on the interval $[2, \infty)$ and $\textrm{supp} \chi \subset [1, \infty)$. Define 
\begin{align*}
&\vec{\breve{g}}_n(r) := \chi(4 r/ \be_n \mu_n)( \vec{\ti{g}}_n(r)- (\pi, 0))\\
& \vec{\breve{b}}_n(r):= \chi(4 r/ \be_n \mu_n) \vec{\ti{b}}_n(r)
\end{align*}
 and observe that we have the following decomposition 
 \begin{align*}
 \vec{\breve{g}}_n(r)  = \vec h_n(0, r) +  \vec{\breve{b}}_n(r) + o_n(1),
 \end{align*}
 where the $o_n(1)$ is in the sense of $ H \times L^2$ -- here we also have used \eqref{chan int}. Moreover,  the right-hand side above, without the $o_n(1)$ term, is a profile decomposition in the sense of \cite[Corollary $2.15$]{CKLS1} because  of Proposition~\ref{no prof} and  \cite[Lemma~$11$]{CKS} or \cite[Lemma~$2.20$]{CKLS1}. We can then consider the nonlinear profiles. Note that by construction we have $\vec{\breve{g}}_n \in \HH_0$ and as in \cite{CKLS1}, we can use \eqref{gn bemu} to show that $\E(\vec{\breve{g}}_n) \le C < 2\E(Q)$ for large $n$. The corresponding wave map evolution $\vec{\breve{g}}_n(t) \in \HH_0$ is thus global in time and scatters as $t \to \pm \infty$ by \cite[Theorem~$1.1$]{CKLS1}. We also need to check that $\E(\vec{\breve{b}}_n) \le C < 2 \E(Q)$. Note that by construction and the definition of $\ti b_n$, we have
 \begin{align*}
\E( \vec{\breve{b}}_n) &\le \E(\vec{\ti{b}}_n) + O\left( \int_{0}^{\infty} \frac{4r^2}{\be_{n, 0}^2 \mu_n^2} (\chi^{\prime})^2( 4r/ \be_n \mu_n) \frac{b_n^2((1-\tau_n)r)}{r} \, dr \right)\\
&\quad  + \int_{\be_n \mu_n/2}^{\be_n \mu_n} \frac{\sin^2( \chi(4r/ \be_n \mu_n) b_{n, 0}((1-\tau_n)r))}{r} \, dr \\
&\le  \E(\vec{\ti{b}}_n) + O\left( \int_{\be_n \la_n/2}^{\be_n \la_n} \frac{b_{n, 0}^2(r)}{r} \, dr \right)\\
&=\E(\vec{\ti{b}}_n) + o_n(1) \le C < 2\E(Q),
\end{align*}
 where the last line follows from  \eqref{bn0 int} since $\be_n \ll \al_n$. 
 
 Arguing as in  \cite{CKLS1}, we can use Proposition~\ref{no prof},  \cite[Proposition~$2.17$]{CKLS1} and \cite[Lemma~$2.18$]{CKLS1} to obtain
 the following nonlinear profile decomposition  
 \begin{align*}
& \vec{\breve{g}}_n(t, r)  = \vec h_n(t, r) +  \vec{\breve{b}}_n(t, r) + \vec{\breve{\theta}}_n(t, r),\\
& \|\vec{\breve{\theta}}_n\|_{L^{\infty}_t(H \times L^2)} \to 0.
 \end{align*}
 Finally observe that by construction and the finite speed of propagation we have 
 \begin{align*}
& \vec{\breve{g}}_n(t, r)  = \vec g_n(t, r) - \pi,\\
 &  \vec{\breve{b}}_n(t, r)  =\vec{\ti{b}}_n (t, r).
 \end{align*}
 for all  $t \in \R$ and $ r \in [\be_n \mu_n/2 + \abs{t}, \infty)$.  Therefore, in particular we have 
 \begin{align*}
 & \vec{g}_n( \be_n \mu_n/2, r) -(\pi, 0) = \vec h_n(\be_n \mu_n/2, r) +  \vec{\ti{b}}_n(\be_n \mu_n/2, r) + \vec{\breve{\theta}}_n( \be_n \mu_n/2, r)
 \end{align*}
 for all $r \in [\be_n \mu_n, \infty)$ which proves \eqref{ghbt}. 

We can combine \eqref{g-Q},  \eqref{ghbt}, \eqref{ti b small before be},  and \eqref{chan int} together with the monotonicity of the energy on interior cones  and the fact that $\| Q( \cdot/ \mu_n) - \pi \|_{H( r \ge \be_n \mu_n)} = o_n(1)$, to obtain the decomposition 
\begin{align}\label{gqhbth}
\vec{g}_n(  \be_n \mu_n/2, r) &=  (Q(r/ \mu_n), 0)+ \vec h_n(\be_n \mu_n/2, r) \\ \notag
& \quad +\vec{\ti{b}}_n( \be_n \mu_n/2, r) + \vec{\ti{\theta}}_n( r),\\
 &\| \vec{\ti{\theta}}_n\|_{H \times L^2} \to 0 .
 \end{align}

%____________________________________figure_____________________________________%

\begin{figure}
\begin{tikzpicture}[
	>=stealth',
	axis/.style={semithick, ->},
	coord/.style={dashed, semithick},
	yscale = 1.2,
	xscale = 1.2]
	\newcommand{\xmin}{0};
	\newcommand{\xmax}{8};
	\newcommand{\ymin}{-0.9};
	\newcommand{\xa}{1};
	\newcommand{\ya}{2};
	\newcommand{\yb}{7.5};
	\newcommand{\xb}{\xa+\yb-\ya};
	\newcommand{\ymax}{2.9};
	\newcommand{\xc}{7.5}
	\newcommand{\delay}{1.5};
	\newcommand{\fsp}{0.2};
	\draw [axis] (\xmin-\fsp,0) node [left] {$t=0$} -- (\xmax,0) node [right] {$r$};
	\draw [axis] (0,\ymin-2*\fsp) -- (0,\ymax) node [below left] {$t$};
	\draw (-\fsp, \ya) node [left] {$\ds t= \frac{\beta_n \mu_n}{2}$} -- (\xmax,\ya);
	\draw (\ya,-0.1) -- (\ya,0.1);
	\draw (\ya,0) node [above left] {$\frac{\beta_n \mu_n}{2}$};
	\draw (3*\ya,-0.1) -- (3*\ya,0.1);
	\draw (3*\ya,0) node [above right] {$\frac{3 \beta_n \mu_n}{2}$};
	\draw [thick, densely dotted, white] (3*\ya+2*\fsp,0) -- (\yb - 2*\fsp,0);
	\draw (\xc,-0.1) -- (\xc,0.1) node [above] {$\alpha_n$};
	\draw (2*\ya, \ya-0.1) -- (2*\ya, \ya+0.1) node [yshift=-18, xshift=2] {$\beta_n \mu_n$};
	\draw [dashed] (\ya,0) -- (2*\ya, \ya) -- (3*\ya,0) -- (3*\ya,\ymin);
	\draw [decorate,decoration={brace,amplitude=10, mirror}] (0, \ymin) -- (3*\ya,\ymin)  node [midway,yshift=-18] {$Q(\cdot/\mu_n)$}; 	
	\draw [decorate,decoration={brace, amplitude=10, mirror, raise=6}] (\ya,0) -- (\xmax-\fsp,0)  node [midway,yshift=-24] {$ h_n(0) + \tilde b_n(0)$};
	\draw [decorate,decoration={brace, amplitude=10, raise=6}] (0, \ya) -- (2*\ya,\ya )  node [midway,yshift=24] {$Q(\cdot/\mu_n)$}; 
	\draw [decorate,decoration={brace, amplitude=10, raise=6}] (2*\ya, \ya ) -- ( (\xmax-\fsp, \ya)  node [midway,yshift=24] {$ h_n(\frac{\beta_n \mu_n}{2}) + \tilde b_n( \frac{\beta_n \mu_n}{2})$}; 	
\end{tikzpicture}
\caption{\label{fig:3g} A schematic description of the evolution of the decomposition \eqref{ghqb} from time $t=0$ until time $t = \frac{\beta_n \mu_n}{2}$. At time $t = \frac{\beta_n \mu_n}{2}$ the decomposition \eqref{gqhbth} holds.}
\end{figure}
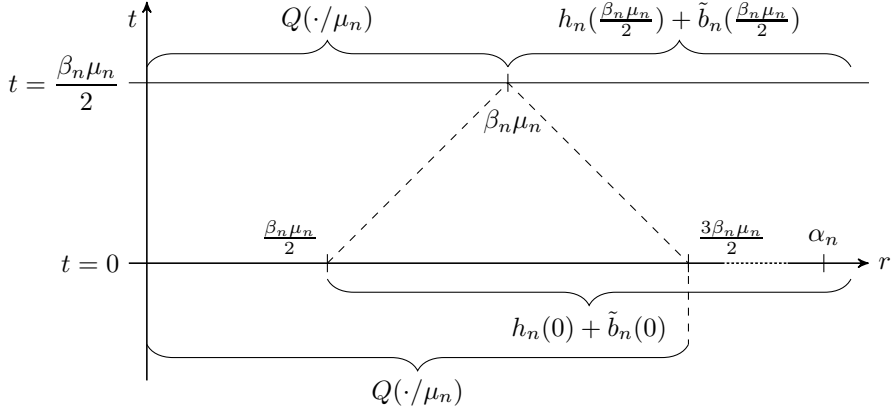

%____________________________________figure_____________________________________%

Now, let $s_n \to \infty$ be any sequence such that $s_n \ge \be_n \mu_n/2$ for each $n$. The next step is to prove the following decomposition at time $s_n$:
\begin{align}\label{ghbz}
&\vec g_n(s_n, r)-( \pi, 0) = \vec h_n(s_n, r) + \vec{\ti{b}}_n(s_n, r) + \vec \zeta_n( r) \quad  \forall r \in [s_n, \infty),\\
&  \|\vec \zeta_n\|_{H \times L^2}  \to 0\label{z to 0} \quad \textrm{as} \quad n \to \infty.
\end{align}
We proceed as in the proof of \eqref{ghbt}. By \eqref{g-Q} we can argue as in Corollary~\ref{tech} and find $\rho_n \to \infty$ with $\rho_n\ll \be_n$ so that 
\begin{align}\label{gn rho}
 g_n(  \be_n \mu_n/2,   \, \rho_n\mu_n ) \to \pi \quad \textrm{as} \quad n  \to \infty.
 \end{align}
Define 
\begin{align*}
&\hat f_{n, 0}(r)=  \begin{cases} \pi - \frac{\pi- g_n( \be_n \mu_n/2,  \,  \rho_n \mu_n)}{ \rho_n \mu_n} r\quad \textrm{if} \quad r \le   \rho_n\mu_n\\ g_n(\be_n \mu_n/2, r) \quad \textrm{if} \quad  r \ge  \rho_n \mu_n  \end{cases}\\
& \hat f_{n, 1}(r)= \dot{g}_n( \be_n\mu_n/2, r).
\end{align*}
Let  $\chi \in C^{\infty}$ be as above and set 
\begin{align*}
&\vec{f}_n(r) := \chi(2r/ \rho_n \mu_n)( \vec{\hat{f}}_n(r)- (\pi, 0)),\\
& \vec{\hat{b}}_n(r):= \chi(2 r/ \rho_n \mu_n) \vec{\ti{b}}_n(\be_n \mu_n/2, r).
\end{align*}
 Observe that we have the following decomposition: 
  \begin{align*}
 \vec{f}_n(r)  = \vec h_n(\be_n \mu_n/2, r) +  \vec{\hat{b}}_n(r) + o_n(1).
 \end{align*}
 where the $o_n(1)$ above is in the sense of $H \times L^2$. Moreover, the  right-hand side above, without the $o_n(1)$ term, is a profile decomposition in the sense of \cite[Corollary~$2.15$]{CKLS1} because  of Proposition~\ref{no prof} and \cite[Lemma $11$]{CKS} or \cite[Lemma~$2.20$]{CKLS1}. We can then consider the nonlinear profiles. Note that by construction we have $\vec{f}_n \in \HH_0$ and, as usual, we can use  \eqref{gn rho} to  show that $\E(\vec{f}_n) \le C < 2\E(Q)$ for large $n$. The corresponding wave map evolution $\vec{f}_n(t) \in \HH_0$ is thus global in time and scatters as $t \to \pm \infty$ by \cite[Theorem~$1.1$]{CKLS1}.

 %____________________________________figure_____________________________________%

  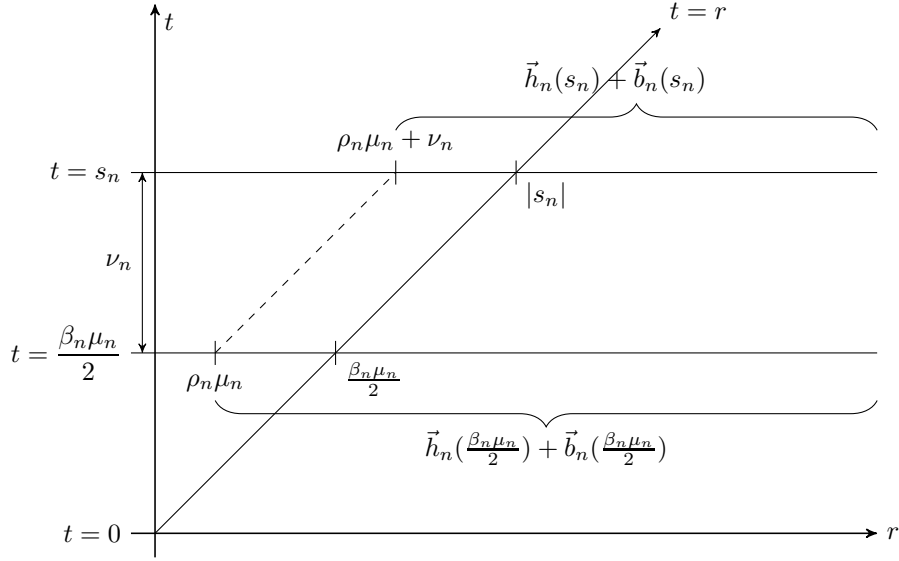
\begin{figure}
\begin{tikzpicture}[
	>=stealth',
	axis/.style={semithick, ->},
	coord/.style={dashed, semithick},
	yscale = 1.6,
	xscale = 1.6]
	\newcommand{\xmin}{-0};
	\newcommand{\xmax}{6};
	\newcommand{\ymin}{-0.2};
	\newcommand{\xa}{0.5}
	\newcommand{\ya}{1.5};
	\newcommand{\yb}{3};
	\newcommand{\xb}{\xa+\yb-\ya}
	\newcommand{\ymax}{4.2};
	\newcommand{\delay}{1.5};
	\newcommand{\fsp}{0.2};
	\draw [axis] (\xmin-\fsp,0) node [left] {$t=0$} -- (\xmax,0) node [right] {$r$};
	\draw [axis] (0,\ymin) -- (0,\ymax+\fsp) node [below right] {$t$};
	\draw [->] (0,0) -- (\ymax,\ymax) node [above right] {$t=r$}; 
	\draw (-\fsp, \ya) node [left] {$\ds t=  \frac{\beta_n \mu_n}{2}$} -- (\xmax,\ya);
	\draw (\xa,\ya-0.1) node [below] {$\rho_n \mu_n$} -- (\xa,\ya+0.1);
	\draw (\ya,\ya+0.1) -- (\ya,\ya-0.1);
	\draw (\ya,\ya)  node [below right] {$\frac{\beta_n \mu_n}{2}$};
	\draw [dashed] (\xa,\ya) -- (\xb,\yb);
	\draw (-\fsp, \yb) node [left] {$\ds t= s_n$} -- (\xmax,\yb);
	\draw (\xb,\yb-0.1) -- (\xb,\yb+0.1) node [above] {$\rho_n \mu_n + \nu_n$};
	\draw (\yb,\yb+0.1) -- (\yb,\yb-0.1);
	\draw (\yb, \yb) node [below right] {$|s_n|$};
	\draw [decorate,decoration={brace, mirror, amplitude=10, raise=18}] (\xa,\ya) -- (\xmax,\ya) node [midway,yshift=-36] {$\vec h_n( \frac{\beta_n \mu_n}{2}) + \vec b_n(\frac{\beta_n \mu_n}{2})$}; 
	\draw  [decorate,decoration={brace, amplitude=10, raise=16}] (\xb,\yb) -- (\xmax,\yb) node [midway,yshift=36, xshift=-8] {$\vec h_n( s_n) + \vec b_n(s_n)$};
	\draw  [<->, xshift=-3] (0, \ya) -- (0, \yb) node [midway,left] {$\nu_n$};	
\end{tikzpicture}
\caption{\label{fig:2g} A schematic depiction of the evolution of the decomposition \eqref{gqhbth} up to time $s_n$. On the interval $[s_n, +\infty)$, the decomposition \eqref{ghbz} holds.}
\end{figure}
 %____________________________________figure_____________________________________%

  As in the proof of  \eqref{ghbt} it is also easy to show that $\E(\vec{\hat{b}}_n) \le C<2\E(Q)$ where here we use \eqref{ti b small before be} instead of \eqref{bn0 int}.

 Again  we can use Proposition~\ref{no prof}, \cite[Proposition~$2.17$]{CKLS1} and \cite[Lemma~$2.18$]{CKLS1}  to obtain
 the following nonlinear profile decomposition 
 \begin{align*}
& \vec{f}_n(t, r)  = \vec h_n(  \be_n \mu_n/2 +t, r) +  \vec{\hat{b}}_n(t, r) + \vec{\ti{\zeta}}_n(t, r),\\
& \|\vec{\ti{\zeta}}_n\|_{L^{\infty}_t(H \times L^2)} \to 0.
 \end{align*}
In particular, for $$\nu_n:= s_n -\be_n \mu_n/2$$ we have 
 \begin{align*}
& \vec{f}_n(\nu_n, r)  = \vec h_n( s_n, r) +  \vec{\hat{b}}_n(\nu_n, r) + \vec{\ti{\zeta}}_n(\nu_n, r).
 \end{align*}
By the finite speed of propagation we have that
\begin{align*}
&\vec{f}_n(\nu_n, r) = \vec g_n(s_n, r) -(\pi, 0), \\
& \vec{\hat{b}}_n(\nu_n, r) = \vec{\ti{b}}_n(s_n, r)
\end{align*}
as long as $r \ge \rho_n \mu_n + \nu_n$. Using the fact that $\rho_n \ll \be_n$ we have that $s_n \ge \rho_n \mu_n + \nu_n$ and hence, 
 \begin{align*}
& \vec{g}_n(s_n, r)-(\pi, 0)  = \vec h_n( s_n, r) +  \vec{\ti{b}}_n(s_n, r) + \vec{\ti{\zeta}}_n(\nu_n, r) \quad \forall r \in [s_n, \infty).
 \end{align*}
 Setting $\vec{\zeta}_n:= \vec{\ti{\zeta}}_n(\nu_n)$ we obtain \eqref{ghbz} and \eqref{z to 0}. With this decomposition we can now complete the proof. 
 
 One the one hand observe that by rescaling, \eqref{da supp}, and the fact that $2\tau_n s_n \ge  \tau_n + \tau_n s_n$ for $n$ large we have 
 \begin{align*}
 \|\vec g_n(s_n)-\vec h_n(s_n)- (\pi, 0)\|_{H \times L^2( r \ge s_n)} &= \| \vec a( \tau_n + \tau_n s_n, \tau_n \cdot)- (\pi, 0)\|_{H \times L^2(r \ge s_n)} \\
 &= \|\vec a( \tau_n + \tau_n s_n)-(\pi, 0)\|_{H \times L^2( r \ge \tau_n s_n)}\\
 & \le  \|\vec a( \tau_n + \tau_n s_n)-(\pi, 0)\|_{H \times L^2( r \ge  \frac{1}{2}(\tau_n+\tau_n s_n))}\\
 & \longrightarrow 0 \quad \textrm{as} \quad n \to \infty.
\end{align*}
Combining the above with the decomposition \eqref{ghbz} and \eqref{z to 0} we obtain that 
\begin{align}\label{ti b to 0 sn}
\|\vec{\ti{b}}_n(s_n)\|_{H \times L^2(r \ge s_n)}  \to 0 \quad \textrm{as} \quad n \to \infty.
\end{align}

On the other hand, combining our assumption \eqref{fail} and Corollary~\ref{ext en est} we know that there exists $\al_0>0$ so that 
\begin{align*}
\|\vec{\ti{b}}_n(s_n)\|_{H \times L^2(r \ge s_n)} &= \|\vec b_n(\tau_n s_n)\|_{H \times L^2(r \ge \tau_n s_n)} \ge \al_0 \de_0.
\end{align*}
But this contradicts \eqref{ti b to 0 sn}. 
\end{proof}

We can  now complete the proof of Theorem~\ref{glob sol}.

 \begin{proof}[Proof of Theorem \ref{glob sol}]
 Let $\vec a(t)$ be defined as in \eqref{a def}.  Recall that by \eqref{ea} we have 
 \begin{align}\label{Epsi-Efy}
 \lim_{t \to  \infty} \E (\vec a(t)) = \E (\vec \psi)- \E(\vec \fy).
 \end{align}
By Proposition~\ref{global} we have found a sequence of times $\tau_n \to \infty$ so that 
 \begin{align*}
 \E(\vec a(\tau_n)) \to \E(Q)
 \end{align*}
 as $n \to \infty$. This then implies that 
 \begin{align*}
 \lim_{t \to \infty} \E (\vec a(t)) =\E(Q).
 \end{align*}
We now use the variational characterization of $Q$ to show that in fact $\|\dot a(t)\|_{L^2} \to 0$ as $t \to  \infty$. To see this observe that since $a(t) \in \HH_1$ we can deduce by \cite[$(2.18)$]{CKLS1} that 
 \begin{align*}
  \E(Q) \leftarrow \E(a(t), \dot{a}(t)) \ge  \int_0^{\infty} \dot{a}^2(t, r) \, r\, dr + \E(Q).
 \end{align*}
 Next observe that the decomposition in \cite[Lemma~$2.5$]{CKLS1} provides us with a function $\la: (0, \infty) \to (0, \infty)$ such that 
 \begin{align*}
 \|a(t,  \cdot)- Q(\cdot/ \la(t))\|_H \le \de(\E( a(t), 0) - \E(Q)) \to 0.
 \end{align*}
 This also implies that 
 \begin{align}\label{a-Q en to 0}
 \E(\vec a(t) - (Q(\cdot/ \la(t)), 0)) \to 0
 \end{align}
 as  $t \to  \infty$. Since $t \mapsto a(t)$ is continuous in $H$ for $t \in[0, \infty)$ it follows from \cite[Lemma~$2.5$]{CKLS1} that $\la(t)$ is continuous on $[0, \infty)$.  Therefore we have established that 
 \begin{align*}
 \vec \psi(t) - \vec \fy(t) - (Q(\cdot/ \la(t)), 0) \to 0 \quad \textrm{in} \quad H \times L^2 \quad \textrm{as} \quad t \to \infty.
 \end{align*}
 It remains to show that $\la(t) = o(t)$. This follows immediately from the asymptotic vanishing  of $\nabla_{t, r} a(t)$ outside the light cone and from \eqref{a-Q en to 0}. To see this observe that by \eqref{da supp} with $\la=1$ we have that $a(t, r)-(\pi, 0) = o(1)$ in $H \times L^2(r \ge t)$ as $t \to \infty$. Therefore we have 
 \begin{align*}
\E_{\frac{t}{\la(t)}}^{\infty}(Q) = \E_{t}^{\infty}( \pi- Q(\cdot/ \la(t))) \le  \E(\vec a(t) - (Q(\cdot/ \la(t)), 0)) +o(1) \to 0
\end{align*}
as $t \to \infty$. But this then implies that $\frac{t}{\la(t)} \to \infty$ as $t \to \infty$. This completes the proof. 
\end{proof}

 \bigskip

\centerline{\scshape Rapha\"el C\^{o}te }
\medskip
{\footnotesize
\begin{center}
CNRS and \'Ecole polytechnique \\
Centre de Math\'ematiques Laurent Schwartz UMR 7640 \\
Route de Palaiseau, 91128 Palaiseau cedex, France \\
\email{cote@math.polytechnique.fr}
\end{center}
} 

\medskip

\centerline{\scshape Carlos Kenig, Andrew Lawrie, Wilhelm Schlag}
\medskip
{\footnotesize
% please put the address of the first author
 \centerline{Department of Mathematics, The University of Chicago}
\centerline{5734 South University Avenue, Chicago, IL 60615, U.S.A.}
\centerline{\email{cek@math.uchicago.edu, alawrie@math.uchicago.edu, schlag@math.uchicago.edu}}
} % Do not forget to end the {\footnotesize by the sign }


\begin{thebibliography}{10}


 
 \bibitem{BG} Bahouri, H., G\'{e}rard, P.  {\em High frequency approximation of solutions to critical nonlinear wave equations}. Amer.\ J.\ Math., 121 (1999), 131--175.
 
\bibitem{Bu} Bulut, A. {\em Maximizers for the Strichartz inequalities for the wave equation}. Differential Integral Equations 23 (2010), no.\ 11-12, 1035--1072.
 
 \bibitem{BKT} Bejenaru, I.,  Krieger, J.,  Tataru, D.  {\em A codimension two stable manifold of near soliton equivariant wave maps}. Preprint 2012  arXiv:1109.3129.
 
  \bibitem{Co} C\^{o}te, R.  {\em Instability of nonconstant harmonic maps for the $(1+2)$-dimensional equivariant wave map system}. Int.\ Math.\ Res.\ Not.\ 2005, no.~57, 3525--3549.
 
 \bibitem{CKM} C\^{o}te, R.,  Kenig, C., Merle, F. {\em 
Scattering below critical energy for the radial 4D Yang-Mills equation and for the $2D$ corotational wave map system.}  
Comm.\ Math.\ Phys.\ 284 (2008), no.~1, 203--225. 
 
 \bibitem{CKLS1} C\^{o}te, R.,  Kenig, C., Lawrie, A.,  Schlag, W. {\em Characterization of large energy solutions of the equivariant wave map problem: I}.   Amer. J. Math. 137 (2015), no. 1, 139--207. (see arxiv:1209.3682 for newest version). 
 
 \bibitem{CKS} C\^{o}te, R.,  Kenig, C., Schlag, W. {\em Energy partition for the linear radial wave equation}. To appear in Math. Ann. Preprint 2012. arXiv:1209.3678 

 \bibitem{CTZ}  Christodoulou, D., Tahvildar-Zadeh, A.\ S. {\em On the regularity of spherically symmetric wave maps}. Comm.\ Pure Appl.\ Math.\ 46 (1993), no.\ 7, 1041--1091.
 
\bibitem{CTZ1} Christodoulou, D., Tahvildar-Zadeh, A.\ S. {\em On the asymptotic behavior of spherically symmetric wave maps}. Duke Math.\ J.\ 71 (1993), no.\ 1, 31--69.
 
 \bibitem{DK} Donninger, R., Krieger, J. {\em Nonscattering solutions and blowup at infinity for the critical wave equation}. Preprint 2012. 	arXiv:1201.3258v1.
 
 \bibitem{DKM1} Duyckaerts, T., Kenig, C.,  Merle, F.  {\em Universality of blow-up profile for small radial type II blow-up solutions of the energy-critical wave equation}. J.\ Eur.\ Math.\ Soc.~(JEMS) 13 (2011), no.~3, 533--599. 
 
  \bibitem{DKM2} Duyckaerts, T., Kenig, C.,  Merle, F.  {\em Universality of the blow-up profile for small type II blow-up solutions of energy-critical wave equation: the non-radial case}.  To appear in J.\ Eur.\ Math.\ Soc.\ (JEMS) 14 (2012) no. 5 1389--1454	
  
    \bibitem{DKM3} Duyckaerts, T., Kenig, C.,  Merle, F.  {\em Profiles of bounded radial solutions of the focusing, energy-critical wave equation}.  Geom. Funct. Anal. 22 \ (2012) no. 3, 639--698.
      
   \bibitem{DKM4} Duyckaerts, T., Kenig, C.,  Merle, F.  {\em Classification of radial solutions of the focusing, energy-critical wave equation}.  Preprint 2012. 	arXiv:1204.0031v1.
   
 \bibitem{G} Grillakis, M. {\em Classical solutions for the equivariant wave maps in $1+2$ dimensions}. Preprint, 1991
 
 \bibitem{Hel} H$\acute{\textrm{e}}$lein, F. {\em Harmonic maps, conservation laws and moving frames}. Translated from the 1996 French original. With a foreword by James Eells. Second edition. Cambridge Tracts in Mathematics, 150. Cambridge University Press, Cambridge, 2002.
 
\bibitem{Kr} Krieger, J. {\em Global regularity and singularity development for wave maps}. Surveys in differential geometry. Vol.\ XII. Geometric flows, 167--201. Surveys in Differential Geometry, 12. International, Somerville, Mass., 2008.
 
 \bibitem{KS} Krieger, J., Schlag, W. {\em Concentration compactness for critical wave maps}. EMS Monographs in Mathematics. European Mathematical Society (EMS), Z\"{u}rich, 2012.  
  
 \bibitem{KST}Krieger, J., Schlag, W., Tataru, D. {\em Renormalization and blow up for charge one equivariant critical wave maps}. Invent.\ Math.\ 171 (2008), no.~3, 543--615.
 
\bibitem{KST2} Krieger, J.,  Schlag, W., Tataru, D. {\em Renormalization and blow up for the critical Yang-Mills problem}. Adv.\ Math.\ 221 (2009), no.~5, 1445--1521.
 
 \bibitem{KM06} Kenig, C.\ E.,  Merle, F.  {\em Global well-posedness, scattering and blow-up for the energy-critical, focusing, non-linear Schr\"odinger equation in the radial case.} Invent.\ Math.\ 166 (2006), no.~3, 645--675. 

 
 \bibitem{KM08} Kenig, C.\ E., Merle, F.  {\em Global well-posedness, scattering and blow-up for the energy-critical focusing non-linear wave equation}. Acta Math.\ 201 (2008), no.~2, 147--212. 
 
 
  \bibitem{LS} Lawrie, A.,  Schlag, W. {\em Scattering for wave maps exterior to a ball}.  Adv.\  Math.\  232 (2013) no.1 57--97.
 
  \bibitem{LinS}  Lindblad, H.,  Sogge, C.\ D. {\em On existence and scattering with minimal regularity for semilinear wave equations.} J.\ Funct.\ Anal.\ 130 (1995), no.~2, 357--426. 
 

\bibitem{RR}  Rapha\"{e}l, P., Rodnianski, I.  {\em Stable blow up dynamics for the critical corotational Wave map and equivariant Yang-Mills problems}.  Publi.\ I.H.E.S., in press.

\bibitem{RS} Rodnianski, I., Sterbenz, J. {\em On the formation of singularities in the critical $O(3)$ $ \sigma$-model}. Ann.\ of Math.\ 172, 187--242 (2010) 

 \bibitem{SU} Sacks, J.,  Uhlenbeck, K. {\em The existence of minimal immersions of $2$-spheres}.  Ann.\ of Math.\ (2) 113 (1981), no.~1, 1--24. 

\bibitem{Sh} Shatah, J. {\em Weak solutions and development of singularities of the $SU(2)$ $\sigma$-model}. Comm.\ Pure Appl.\ Math.\ 41 (1988), no.~4, 459--469.
 
 \bibitem{SS} Shatah, J.,  Struwe, M.  {\em Geometric wave equations}. Courant Lecture Notes in Mathematics, 2. New York University, Courant Institute of Mathematical Sciences, New York; American Mathematical Society, Providence, RI, 1998.  
 
   \bibitem{SS2} Shatah, J.,  Struwe, M. {\em The Cauchy problem for wave maps.} Int.\ Math.\ Res.\ Not.~2002, no.~11, 555--571.
 
 
 \bibitem{STZ} Shatah, J., Tahvildar-Zadeh, A.\ S. {\em On the Cauchy problem for equivariant wave maps}. Comm.\ Pure Appl.\ Math. 47 (1994), no.~5, 719--754.

\bibitem{Sogge}  Sogge, C.\ D.  {\em Lectures on non-linear wave equations.} Second edition. International Press, Boston, MA, 2008.

 \bibitem{ST1} Sterbenz, J.,  Tataru, D.  {\em Energy dispersed large data wave maps in $2+1$ dimensions}. Comm.\ Math.\ Phys.\ 298 (2010), no.~1, 139--230
 
 \bibitem{ST2}Sterbenz, J.,  Tataru, D.   {\em Regularity of wave-maps in dimension $2+1$}. Comm.\ Math.\ Phys.\ 298 (2010), no.~1, 231--264.
 

  \bibitem{St} Struwe, M. {\em Equivariant wave maps in two space dimensions}. Comm.\ Pure Appl.\ Math.\ 56 (2003), no.~7, 815--823.

 
 \bibitem{T2} Tao, T.  {\em Global regularity of wave maps II. Small energy in two dimensions}. Comm.\ Math.\ Phys.\ 224 (2001), no.~2, 443--544. 

\bibitem{T3}Tao, T. {\em Global regularity of wave maps III-VII}. Preprints 2008-2009.

\bibitem{Tat} Tataru, D.   {\em On global existence and scattering for the wave maps equation}. Amer.\ J.\ Math.\ 123 (2001), no.~1, 37--77. 

  \end{thebibliography}
\end{document}